\documentclass[12pt]{amsart}
\usepackage{amsmath,amsfonts,amssymb,amsthm,amstext,pgf,graphicx,hyperref,verbatim,lmodern,textcomp,color,young,tikz}
\usetikzlibrary{decorations}
\usepackage[mathscr]{euscript}
\usetikzlibrary{decorations.markings}
\usetikzlibrary{arrows}
\usepackage{float}
\theoremstyle{plain}
\newtheorem{theorem}{Theorem}[section]
\newtheorem{lemma}[theorem]{Lemma}
\newtheorem{proposition}[theorem]{Proposition}
\newtheorem{corollary}{Corollary}

\theoremstyle{definition}
\newtheorem{defn}{Definition}[section]

\title{}
\begin{document}
\title[Line graph characterization of power graphs of finite groups]{Line graph characterization of power graphs of finite nilpotent groups}
\author[Sudip Bera]{Sudip Bera}
\address[Sudip Bera]{Department of Mathematics, Indian Institute of Science, Bangalore 560 012}
\email{sudipbera517@gmail.com}
\keywords{Nilpotent group; Line graph; Power graph; Proper power graph.}
\subjclass[2010]{05C25; 20D10}
\maketitle
\begin{abstract}
This paper deals with the classification of groups $G$ such that power graphs and proper power graphs of $G$ are line graphs. In fact, we classify all finite nilpotent groups whose power graphs are line graphs. Also, we categorize all finite nilpotent groups (except non abelian $2$-groups) whose proper power graphs are line graphs. Moreover, we innvestigate when the proper power graphs of generalized quaternion groups are line graphs. Besides, we derive a condition on the order of the dihedral groups for which the proper power graphs of the dihedral groups are line graphs.   
\end{abstract}
\section{Introduction}
\label{sec:intro}
The investigation of graph representations is one of the interesting and popular research topic in algebraic graph theory, as graphs like these enrich both algebra and graph theory. Moreover, they have important applications (see, for example, 
\cite{surveypwrgraphkac1, cayleygraphsckry}) and are related to automata theory \cite{automatatheory}. During the last two decades, the investigation of the interplay
between the properties of an algebraic structure $S$ and the graph-theoretic properties of $\Gamma(S),$ a graph
associated with $S,$ has been an exciting topic of research. Different types of graphs, specifically power graph of semigroup \cite{undpwrgraphofsemgmainsgc1, directedgrphcompropofsemgrpkq3}, group \cite{combinatorialpropertyandpowergraphsofgroupskq1}, normal subgroup based power graph of group \cite{normalsubgrpbasedpwrbb2}, intersection power graph of group \cite{intersectionpwegraphb3}, enhanced power graph of group \cite{firstenhcedpwrstrctreaacns1,bera-dey-mukherjee-connectivity-enhanced} etc. have been introduced to study algebraic structures using graph theory. One of the major graph representation
amongst them is the power graphs of finite groups. We found several papers in this context \cite{surveypwrgraphkac1,pwrgraphoffntgrpgc1,Heidar-Jafari,forbidden-cameron-manna-Mehatari,undpwrgraphofsemgmainsgc1,Hamzeh-ashrafi}.
The concept of a power graph was introduced in \cite{combinatorialpropertyandpowergraphsofgroupskq1}. As
explained in the survey \cite{surveypwrgraphkac1}, this definition also covered the undirected graphs. Accordingly, the present paper follows Chakrabarty \emph{et al}. and uses the brief term ``power graph'' defined as follows. 
\begin{defn}[\cite{surveypwrgraphkac1,undpwrgraphofsemgmainsgc1, combinatorialpropertyandpowergraphsofgroupskq1}]\label{defn: powr graph}
Let $S$ be a semigroup, then the \emph{power graph} $\mathcal{P}(S)$ of $S,$ is a simple graph, whose vertex set is $S$ and two distinct vertices $u$ and $v$ are edge connected if and only if either $u^m=v$ or $v^n=u,$ where $m, n\in \mathbb{N}.$    
\end{defn}
The authors in \cite{undpwrgraphofsemgmainsgc1} studied various properties of the power graph. They characterized the class of semigroups $S$ for which $\mathcal{P}(S)$ is connected or complete. As a consequence they proved the following; 
\begin{lemma}[Theorem 2.12, \cite{undpwrgraphofsemgmainsgc1}]\label{thm:P(G) compltele iff G cylcic p group}
For a  finite group $G,$ the power graph $\mathcal{P}(G)$ is complete if and only if $G$ is cylcic group of order $1$ or $p^m,$ for some prime $p$ and for some $m\in \mathbb{N}.$ 
\end{lemma}
It is clear that for two groups $G_1$ and $G_2, G_1\cong G_2$ implies that $\mathcal{P}(G_1)\cong\mathcal{P}(G_2).$ A natural question that arises is: Does the converse hold? In \cite{pwrgraphoffntgrpgc1}, the authors showed that the non-isomorphic finite groups may have isomorphic power graphs, but that finite abelian groups with isomorphic power graphs must be isomorphic. Also they conjectured that two finite groups with isomorphic power graphs have the same number of elements of each order. Then Cameron proved this conjecture in \cite{jgt-cameron}.

In \cite{curtin-Pourgholi-prpoer-power-graph}, Curtin \emph{et al}. introduced the concept of deleted power graphs. They introduced deleted power graph as follows;
\begin{defn}\label{defn: proper enhacd pwr graph}
Given a group $G,$ the \emph{proper power graph} of $G,$ denoted by $\mathcal{P}^{**}(G),$ is the graph obtained by deleting all the dominating vertices from the power graph $\mathcal{P}(G).$ Moreover, by $\mathcal{P}^{*}(G)$ we denote the graph obtained by deleting only the identity element of $G$ and this is called \emph{deleted power graph} of $G.$ Note that if there is no such dominating vertex other than identity, then $\mathcal{P}^{*}(G)=\mathcal{P}^{**}(G).$
\end{defn}
Curtin \emph{et al}. discussed the diameter of the proper power graph of the symmetric group $S_n$ on $n$ symbols. For more information related to proper power graphs we refer to \cite{curtin-Pourgholi-prpoer-power-graph, Dostabadi-Farrokhi-Ghouchan,Shi}. Then Aalipour \emph{et al}. in \cite[Question 40]{firstenhcedpwrstrctreaacns1} asked about the connectivity of proper power graphs. Recently, Cameron and Jafari in \cite{Heidar-Jafari} answered this question. Also they characterized all dominatable power graphs by following lemma:
\begin{lemma}[Theorem 4, \cite{Heidar-Jafari}] Let $G$ be a finite group. Suppose that $x\in G$ has the property that for all $y\in G,$ either $x$ is a power of $y$ or vice versa. Then one of the following holds:
\begin{enumerate}\label{lemma: classification of dominating vertices of power graph}
\item[(a)]
$x=e;$
\item [(b)]
$G$ is cyclic and $x$ is a generator;
\item[(c)] 
$G$ is a cyclic p-group for some prime $p$ and $x$ is arbitrary;
\item[(d)]
$G$ is a generalized quaternion group and $x$ has order $2.$
\end{enumerate}	
\end{lemma} 
\subsection{Basic Definitions, Notations and Main Results}
We begin this section with some standard definitions from graph theory and group theory. For the convenience of the reader and also for later use, we recall some basic
definitions and notations about graphs. 
Let $\Gamma=(V, E)$ be a graph where $V$ is the set of vertices and $E$ is the set of edges. A graph $\Gamma'=(V', E')$ is a subgraph of another graph $\Gamma=(V, E)$ if and only if $V'\subset V,$ and $E'\subset E.$ An \emph{induced} subgraph of a graph is another graph, formed from a subset of the vertices of the graph and all of the edges connecting pairs of vertices in that subset. A graph $\Gamma$ is said to be \emph{connected} if for any pair of vertices $u$ and $v,$ there exists a path between $u$ and $v.$ $\Gamma$ is said to be \emph{complete} if any two distinct vertices are adjacent. A \emph{clique} of a graph $\Gamma$ is an induced subgraph of $\Gamma$ that is complete. The complete graph with $n$ vertices is denoted by $K_n.$ A \emph{bipartite} graph (or bigraph) is a graph whose vertices can be divided into two disjoint and independent sets $V_1$ and $V_2$ such that every edge connects a vertex in $V_1$ to one in $V_2.$ Vertex sets $V_1$ and $V_2$ are usually called the parts of the graph. A \emph{complete bipartite} graph or biclique is a special kind of bipartite graph where every vertex of the first set is connected to every vertex of the second set. The \emph{star} graph with $n+1$ vertices is denoted by $\Gamma_{1, n}$ which consists of a single vertex with $n$ neighbours. A star graph with the vertex set $v, v', v'', v'''$ is denoted by $\Gamma_{1, 3}(v, v', v'', v'''),$ where $v$ is edge connected to each of the vertices $v', v'', v'''$ and there is no edge between the vertices $v', v'', v'''.$ A vertex of a graph $\Gamma=(V, E)$ is called a \emph{dominating vertex} if it is adjacent to every other
vertex. For a graph $\Gamma,$ let $\text{Dom}(\Gamma)$ denote the set of all dominating vertices in $\Gamma.$ The \emph{vertex connectivity} of a graph $\Gamma,$ denoted by $\kappa{(\Gamma)}$ is the minimum number of vertices which need to be removed from the vertex set $\Gamma$ so that the
induced subgraph of $\Gamma$ on the remaining vertices is disconnected. The complete graph with $n$ vertices has connectivity $n-1.$ A graph $\Gamma$ is a \emph{cograph }if it has no induced subgraph isomorphic to the four-vertex path $P_4.$	A graph $\Gamma$ is \emph{chordal} if it contains no induced cycles of length greater than $3;$ in other words, every cycle on more than $3$ vertices has a chord. A \emph{threshold} graph is a graph containing no induced subgraph isomorphic
to $P_4 , K_4$ or $2K_2 (\text{ or } K_2\bigoplus K_2)\text{ (two disjoint edges with no further edges connecting them) }.$ In general, let $\Gamma_1, \cdots, \Gamma_m$ be $m$ graphs such that $V(\Gamma_i)\cap V(\Gamma_j)=\emptyset,$ for $i\neq j.$ Then $\Gamma=\Gamma_1\bigoplus\cdots\bigoplus\Gamma_m$ be a graphs with vertex set is $V(\Gamma)=V(\Gamma_1)\cup\cdots\cup V(\Gamma_m)$ and $E(\Gamma)=E(\Gamma_1)\cup\cdots\cup E(\Gamma_m).$ Two graphs $\Gamma_1$ and $\Gamma_2$ are \emph{isommphic} if there is a
bijection, $f$ (say) from $V(\Gamma_1)$ to $V(\Gamma_2)$ such that $v\sim v'$ in $\Gamma_1$ if and only if
$f(v)\sim f(v')$ in $\Gamma_2.$ for the vertices $v, v', v\sim v'$ denotes that $v$ and $v'$ are edge connected. Also $v\nsim v'$ means that $v$ and $v'$ are not edge connected.
\begin{defn}[\cite{Godsil,w}]
The line graph of a graph $\Gamma$ is the graph $L(\Gamma)$ with the edges of $\Gamma$ as its vertices, and where two edges of $\Gamma$ are adjacent in $L(\Gamma)$ if and only if they are incident in $\Gamma.$ If a graph $\Gamma'$ is a line graph of some graph then we can call the graph $\Gamma'$ a line graph.  	
\end{defn}
One of the most important results related to the characterization of line graph is Lemma \ref{line graph}. For more information on line graphs we refer \cite{Godsil,w}.
\begin{lemma}[Theorem 7.1.18, \cite{Godsil}]\label{line graph}
A graph $\Gamma$ is the line graph of some graph if and only if $\Gamma$ does not have any of the nine graphs in Figure \ref{fig:line grapph theory} as an induced subgraph.
\end{lemma}
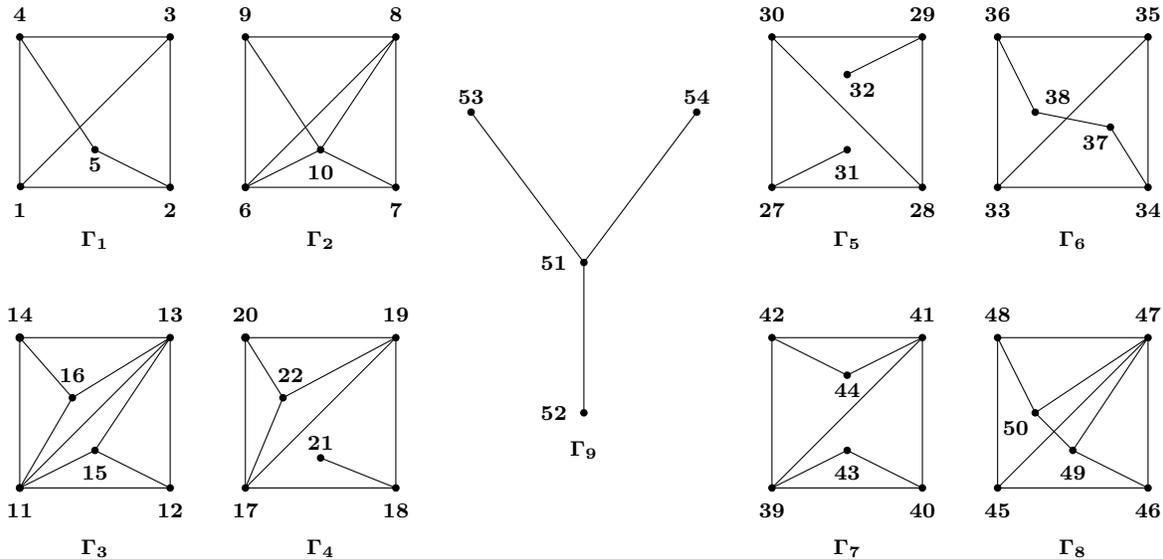
\begin{figure}[H]
	\tiny
	\centering
	\begin{tikzpicture}[scale=1]
	\tikzstyle{edge_style} = [draw=black, line width=2mm, ]
	\draw  (0,0) rectangle (2,2);
	\draw  (0,0) -- (2,2);
	\draw  (2,0) -- (1,.5);
	\draw  (0,2) -- (1,.5);
	\node (e) at (0,-.3)  {$\bf{1}$};
	\node (e) at (2,-.3)  {$\bf{2}$};
	\node (e) at (2,2.3)  {$\bf{3}$};
	\node (e) at (0,2.3)  {$\bf{4}$};
	\node (e) at (1,.3)  {$\bf{5}$};
	\node (e) at (1,-.7)  {$\bf{\Gamma_1}$};
	%
	\fill[black!100!] (0.01, 0.01) circle (.05);
	\fill[black!100!] (2, 2) circle (.05);
	\fill[black!100!] (2, 0) circle (.05);
	\fill[black!100!] (0,2) circle (.05);
	\fill[black!100!] (1,.5) circle (.05);
	%
	\draw  (3,0) rectangle (5,2);
	\draw (3,0)--(5,2);
	\draw (4,.5)--(5,2);
	\draw (4,.5)--(3,0);
	\draw (4,.5)--(5,0);
	\draw (4,.5)--(3,2);
	\node (e) at (3,-.3){$\bf{6}$};
	\node (e) at (5,-.3){$\bf{7}$};
	\node (e) at (5,2.3){$\bf{8}$};
	\node (e) at (3,2.3){$\bf{9}$};
	\node (e) at (4,.2){$\bf{10}$};
	\node (e) at (4,-.7)  {$\bf{\Gamma_2}$};
	\fill[black!100!] (3,0) circle (.05);
	\fill[black!100!] (5, 2) circle (.05);
	\fill[black!100!] (4, 0.5) circle (.05);
	\fill[black!100!] (3,2) circle (.05);
	\fill[black!100!] (3,0) circle (.05);
	\fill[black!100!] (5,0) circle (.05);
	\draw  (0,-4) rectangle (2,-2);
	\draw (0,-4)--(2,-2);
	\draw (0,-4)--(1,-3.5);
	\draw (2,-4)--(1,-3.5);
	\draw (2,-2)--(1,-3.5);
	\draw (.7,-2.8)--(0,-2);
	\draw (.7,-2.8)--(2,-2);
	\draw (.7,-2.8)--(0,-4);
	\node (e) at (0,-4.3){$\bf{11}$};
	\node (e) at (2,-4.3){$\bf{12}$};
	\node (e) at (2,-1.7){$\bf{13}$};
	\node (e) at (0,-1.7){$\bf{14}$};
	\node (e) at (1,-3.8){$\bf{15}$};
	\node (e) at (.7,-2.5){$\bf{16}$};
	\node (e) at (1,-4.8)  {$\bf{\Gamma_3}$};
	\fill[black!100!] (0,-4) circle (.05);
	\fill[black!100!] (2, -2) circle (.05);
	\fill[black!100!] (1, -3.5) circle (.05);
	\fill[black!100!] (2,-4) circle (.05);
	\fill[black!100!] (0,-4) circle (.05);
	\fill[black!100!] (.7,-2.8) circle (.05);
	\filldraw[black!100] (0,-2) circle (.05);
	\draw  (3,-4) rectangle (5,-2);
	\draw  (3,-4) -- (5,-2);
	\draw  (5,-4) -- (4,-3.6);
	\draw  (3,-2) -- (3.5,-2.8);
	\draw  (5,-2) -- (3.5,-2.8);
	\draw  (3,-4) -- (3.5,-2.8);
	\node (e) at (3,-4.3){$\bf{17}$};
	\node (e) at (5,-4.3){$\bf{18}$};
	\node (e) at (5,-1.7){$\bf{19}$};
	\node (e) at (3,-1.7){$\bf{20}$};
	\node (e) at (4,-3.4){$\bf{21}$};
	\node (e) at (3.6,-2.5){$\bf{22}$};
	\node (e) at (4,-4.8)  {$\bf{\Gamma_4}$};
	\fill[black!100!] (3,-4) circle (.05);
	\fill[black!100!] (5, -2) circle (.05);
	\fill[black!100!] (4, -3.6) circle (.05);
	\fill[black!100!] (5,-4) circle (.05);
	\fill[black!100!] (0,-4) circle (.05);
	\fill[black!100!] (3.5,-2.8) circle (.05);
	\filldraw[black!100] (3,-2) circle (.05);
	\draw  (10,0) rectangle (12,2);
	\draw  (12,0) -- (10,2);
	\draw  (10,0) -- (11,.5);
	\draw  (12,2) -- (11,1.5);
	\node (e) at (10,-.3){$\bf{27}$};
	\node (e) at (12,-.3){$\bf{28}$};
	\node (e) at (12,2.3){$\bf{29}$};
	\node (e) at (10,2.3){$\bf{30}$};
	\node (e) at (11,.2){$\bf{31}$};
	\node (e) at (11.2,1.3){$\bf{32}$};
	\node (e) at (11,-.7)  {$\bf{\Gamma_5}$};
	\fill[black!100!] (10,0) circle (.05);
	\fill[black!100!] (12, 0) circle (.05);
	\fill[black!100!] (12, 2) circle (.05);
	\fill[black!100!] (10,2) circle (.05);
	\fill[black!100!] (11,.5) circle (.05);
	\fill[black!100!] (11,1.5) circle (.05);
	%
	\draw  (13,0) rectangle (15,2);
	\draw  (13,0) -- (15,2);
	\draw  (15,0)-- (14.5,.8);
	\draw  (13,2) -- (13.5,1);
	\draw  (14.5,.8) -- (13.5,1);
	\node (e) at (13,-.3){$\bf{33}$};
	\node (e) at (15,-.3){$\bf{34}$};
	\node (e) at (15,2.3){$\bf{35}$};
	\node (e) at (13,2.3){$\bf{36}$};
	\node (e) at (14.3,.6){$\bf{37}$};
	\node (e) at (13.8,1.2){$\bf{38}$};
	\node (e) at (14,-.7)  {$\bf{\Gamma_6}$};
	\fill[black!100!] (13,0) circle (.05);
	\fill[black!100!] (15,2) circle (.05);
	\fill[black!100!] (15,0) circle (.05);
	\fill[black!100!] (13,2) circle (.05);
	\fill[black!100!] (14.5,.8) circle (.05);
	\fill[black!100!] (13.5,1) circle (.05);
	%
	\draw  (10,-4) rectangle (12,-2);
	\draw (10,-4)--(12,-2);
	\draw (10,-4)--(11,-3.5);
	\draw (12,-4)--(11,-3.5);
	\draw (10,-2)--(11,-2.5);  
	\draw (12,-2)--(11,-2.5);		
	\node (e) at (10,-4.3){$\bf{39}$};
	\node (e) at (12,-4.3){$\bf{40}$};
	\node (e) at (12,-1.7){$\bf{41}$};
	\node (e) at (10,-1.7){$\bf{42}$};
	\node (e) at (11,-3.8){$\bf{43}$};
	\node (e) at (11,-2.7){$\bf{44}$};
	\node (e) at (11,-4.8)  {$\bf{\Gamma_7}$};
	\fill[black!100!] (10,-4) circle (.05);
	\fill[black!100!] (12, -2) circle (.05);
	\fill[black!100!] (11, -3.5) circle (.05);
	\fill[black!100!] (12,-4) circle (.05);
	\fill[black!100!] (10,-2) circle (.05);
	\fill[black!100!] (11,-2.5) circle (.05);
	%
	\draw  (13,-4) rectangle (15,-2);
	\draw  (13,-4) -- (15,-2);
	\draw (15,-4)--(14,-3.5);
	\draw (13,-2)--(13.5,-3);
	\draw (15,-2)--(13.5,-3);
	\draw (15,-2)--(14,-3.5);
	\draw (13.5,-3)--(14,-3.5);
	\node (e) at (13,-4.3){$\bf{45}$};
	\node (e) at (15,-4.3){$\bf{46}$};
	\node (e) at (15,-1.7){$\bf{47}$};
	\node (e) at (13,-1.7){$\bf{48}$};
	\node (e) at (14,-3.8){$\bf{49}$};
	\node (e) at (13.24,-3.2){$\bf{50}$};
	\node (e) at (14,-4.8)  {$\bf{\Gamma_8}$};
	\fill[black!100!] (13,-4) circle (.05);
	\fill[black!100!] (15, -2) circle (.05);
	\fill[black!100!] (14, -3.5) circle (.05);
	\fill[black!100!] (15,-4) circle (.05);
	\fill[black!100!] (13,-2) circle (.05);
	\fill[black!100!] (13.5,-3) circle (.05);
	%
	\draw (7.5,-1)--(7.5,-3);
	\draw (7.5,-1)--(6,1);
	\draw (7.5,-1)--(9,1);
	\node (e) at (7.1,-1){$\bf{51}$};
	\node (e) at (7.1, -3){$\bf{52}$};
	\node (e) at (6,1.2){$\bf{53}$};
	\node (e) at (9,1.2){$\bf{54}$};
	\node (e) at (7.5,-3.5)  {$\bf{\Gamma_9}$};
	\fill[black!100!] (7.5,-1) circle (.05);
	\fill[black!100!] (7.5, -3) circle (.05);
	\fill[black!100!] (6,1) circle (.05);
	\fill[black!100!] (9,1) circle (.05);
	\end{tikzpicture}
	\caption{In this figure there are nine graphs namely, $\Gamma_1,\Gamma_2, \Gamma_3, \Gamma_4, \Gamma_5, \Gamma_6, \Gamma_7,\Gamma_8, \Gamma_9 $ and the numbers that appear in this figure are the vertices of the corresponding graphs. }
	\label{fig:line grapph theory}	
\end{figure}
For more information on the graph theory
we refer to \cite{Bon,Godsil, w}. 

Throughout this paper we consider $G$ as a finite group.  $|G|$ denotes the cardinality of the set $G.$ For a prime $p,$ a group $G$ is said to be a $p$-group if $|G|=p^{r}, r\in \mathbb{N}.$ Recall that a finite group $G$ is nilpotent if and only if it is a direct product of its Sylow $p$-subgroups over primes p dividing $|G|.$ Note that, in a nilpotent group, elements of different prime orders commute. For more information on nilpotent groups we refer to \cite{Hunger,robinsongroup,scott-group}.
\begin{lemma}[Proposition 7.5, \cite{Hunger}]\label{Nilpotent group charecterization thm}
A finite group is nilpotent if and only if it is the direct product of its Sylow subgroups.
\end{lemma}
\begin{lemma}[Corollary 7.6, \cite{Hunger}]\label{m|G, nilpotent G has subgroup of ordder m}
If $G$ is a finite nilpotent group and $m$ divides $|G|,$ then $G$ has a subgroup of order $m.$
\end{lemma}
 \begin{lemma}[Theorem 5.4.10, \cite{Gorenstein-group-book}]\label{p group unique subgrp of order p, g is cyclic}
If $G$ is a $p$-group with a unique subgroup of order $p$ for an odd prime $p,$ then G is cyclic.
 \end{lemma}
For any element $g \in G, \text{o}(g)$ denotes the order of the element $g \in G.$ Throughout this article $v^{(m)}$ denotes a vertex of order $m$ of  $G.$ Let $m$ and $n$ be any two positive integers, then the greatest common divisor of $m$ and $n$ is denoted by $\text{gcd}(m, n).$ The \emph{exponent} of a group $G,$ denoted by $\text{exp}(G)$ is defined as the least common multiple of the orders of all elements of the group. If there is no least common multiple, the exponent is taken to be infinity (or sometimes zero, depending on the convention). The Euler's phi function $\phi(n)$ is the number of integers $k$ in the range $1 \leq k \leq n$ for which the  $\text{gcd}(n, k)$ is equal to $1.$ The set $\{1, 2, \cdots, n\}$ is denoted by $[n].$ Throughout this paper, the group operation of any abelian group is taken to be additive. 

A number of important graph classes, including line graphs, cographs, chordal graphs, split graphs, and threshold graphs, can be defined either structurally or in terms of forbidden induced subgraphs. Recently, Cameron, Manna and Mehatari in \cite{forbidden-cameron-manna-Mehatari}, determined completely the groups whose power graph is a threshold or split graph. Moreover, they determined completely the finite nilpotent groups whose power graph is a cograph. Motivated by this work, in this paper we give attention is on the following problems:
\begin{itemize}
\item
Characterize all finite groups $G$ such that  $\mathcal{P}(G)$ is a line graph of some graph $\Gamma.$ 
\item
Characterize all finite groups $G$ such that  $\mathcal{P}^{**}(G)$ is a line graph of some graph $\Gamma.$ 
\end{itemize}
We obtain the following two main results.
\begin{theorem}\label{Thm:P(G) is line graph for nilpotent group} Let $G$ be a nilpotent group. Then $\mathcal{P}(G)$ is a line graph of some graph $\Gamma$ if and only if $G$ is cyclic $p$-group.	
\end{theorem}
\begin{theorem}\label{Thm:P**(G) is line graph for nilpotent group} Let $G$ be a nilpotent group (except non abelian $2$-groups). Then $\mathcal{P}^{**}(G)$ is a line graph of some graph $\Gamma$ if and only if $G$ is one of the following:
\begin{enumerate}
\item[(a)] 
$G\cong\mathbb{Z}_{p^t}, t\geq 1$
\item[(b)]
$G\cong \mathbb{Z}_{pq}$
\item[(c)]
$G\cong \mathbb{Z}_2\times \mathbb{Z}_{2^{2}}$
\item[(d)]	
$G\cong \mathbb{Z}_{2^2}\times \mathbb{Z}_{2^2}$
\item[(e)]
$G\cong\underbrace{\mathbb{Z}_p\times \cdots\times \mathbb{Z}_p}_{k \text{ times,  } k\geq 2}$
\item[(f)]
$G \text{ is non abelian  } p \text{ group and } G=\mathbb{Z}_{p^{t_1}}\cup \cdots\cup \mathbb{Z}_{p^{t_{\ell}}},$ where $\ell$ is the number of distinct subgroups of order $p.$ 
\end{enumerate}		
\end{theorem}
If one takes $a, b, c$ in $\mathbb{Z}_p$ for an odd prime $p,$ then one has the \emph{Heisenberg group} modulo $p.$ It is a group of order $p^3$ with generators $x, y$ and relations:
\[z=xyx^{-1}y^{-1}, x^{p}=y^{p}=z^{p}=1, xz=zx, yz=zy.\] 
\begin{corollary}
Let $G$ be the Heisenberg group modulo $p.$ Then $\mathcal{P}^{**}(G)$ is a line graph. 	
\end{corollary}
\begin{proof}
In the Heisenberg group modulo $p,$ the order of each element is $p.$ Therefore by the part (f) of Theorem \ref{Thm:P**(G) is line graph for nilpotent group}, $\mathcal{P}^{**}(G)$ is a line graph. 
\end{proof}	
\begin{corollary}
Let $G$ be a non abelian group such that $\text{exp}(G)=p.$ Then $\mathcal{P}^{**}(G)$ is a line graph.	
\end{corollary}
\begin{proof}
$G$ is non abelian group and $\text{exp}(G)=p.$ So $G$ is non abelian $p$-group and order of each element of $G$ is $p.$ Hence by by the part (f) of Theorem \ref{Thm:P**(G) is line graph for nilpotent group}, $\mathcal{P}^{**}(G)$ is a line graph.  	
\end{proof}	
Moreover, for non-abelian 2-groups, we prove two theorems.
\begin{theorem}\label{THm: Line graph of generalized quaternion group}
Let $Q_{2^n}$ be the generalized quaternion group. Then $\mathcal{P}^{**}(Q_{2^n})$ is a line graph.	
\end{theorem}
\begin{theorem}\label{thm: P**(D_n) is a line graph if and only}
Let $D_n$ be the dihedral group of order $2n, (n\geq 3).$ Then $\mathcal{P}^{**}(D_n)$ is a line graph if and only if $n=2^k, k\in \mathbb{N}.$
\end{theorem}
Now, let us briefly summarize the content. In Section $2,$ we consider the power graphs of finite nilpotent groups and classify all those power graphs which are line graphs. In Section $3,$ we focus on the proper power graphs of finite nilpotent groups (except non abelian $2$-groups) and characterize all those proper power graphs which are line graphs. Moreover, in this section, we study that the proper power graph of generalized quaternion group is line graph. Also we derive the condition on the order of the dihedral group such that the proper power graph is line graph. 
\section{Proof of main theorem for the power graphs}
Here we give the proof of Theorem \ref{Thm:P(G) is line graph for nilpotent group}. To prove this theorem we need to go through three theorems. The first theorem describes the cyclic group case.
\begin{theorem}
Let $G$ be a finite cyclic group. Then there exists a graph $\Gamma$ such that $\mathcal{P}(G)=L(\Gamma)$ if and only if $G$ is a $p$-group.
\end{theorem}
\begin{proof}
Let $G$ be a cyclic $p$-group with $|G|=p^{r}.$ Then by Lemma \ref{thm:P(G) compltele iff G cylcic p group}, $\mathcal{P}(G)$ is complete. As a result the star graph $\Gamma_{1, p^r}$ serves the purpose.

Conversely, let $|G|$ has at least two distinct prime factors say $p, q.$ Now $G$ is cyclic, so $G$ has a unique subgroup $H$ of order $pq.$ Again $\phi(pq)\geq 2$ implies that $H$ has at least two elements namely $v^{(pq)}_1, v^{(pq)}_2$ of order $pq.$ Again the cyclic subgroup $H= \langle v^{(pq)}_1\rangle=\langle v^{(pq)}_1\rangle$ has elements $v^{(p)}$ and $v^{(q)}$ (say) of order $p$ and $q$ respectively. Now we replace the vertices of the graph $\Gamma_2$ in Figure \ref{fig:line grapph theory} in the following way:\[ 6 \text{ by } v^{(pq)}_1, 10 \text{ by } e, \text{ the identity of } G, 8 \text{ by }v^{(pq)}_2, 7 \text{ by } v^{(p)} \text{ and } 9 \text{ by } v^{(q)}.\] Then the resulting graph is also isomorphic to the graph $\Gamma_2$ in the Figure \ref{fig:line grapph theory}. Therefore, $\mathcal{P}(G)$ contains an induced subgraph isomorphic to $\Gamma_2.$ Hence the theorem.  
\end{proof}
The next theorem handles noncyclic abelian groups. 
\begin{theorem}\label{Thm:P(G) is not line graph, non cylcic abelian grp}
Let $G$ be a non-cyclic abelian group. Then there does not exists any graph $\Gamma$ such that $\mathcal{P}(G)=L(\Gamma).$ 	
\end{theorem}
\begin{proof}
We prove this theorem by showing that the power graph $\mathcal{P}(G)$ has no induced subgraph isomorphic to the graph $\Gamma_{1, 3}.$	
Now it is given that $G$ is non cyclic abelian. So, \[G\cong\mathbb{Z}_{p^{t_{11}}_1}\times \cdots\times\mathbb{Z}_{p^{t_{1k_1}}_1}\times \mathbb{Z}_{p^{t_{21}}_2}\times\cdots\times\mathbb{Z}_{p^{t_{2k_2}}_2}\times\cdots\times
\mathbb{Z}_{p^{t_{r1}}_r}\times\cdots\times\mathbb{Z}_{p^{t_{rk_r}}_r},\] where $k_i\geq1, 1\leq t_{i1}\leq t_{i2}\leq\cdots\leq t_{ik_i}, $ for all $i\in [r]$ and there exists at least one $k_i$ such that $k_i\geq 2$ (if each $k_i=1,$ then $G$ would be cyclic). Without loss of generality, we assume that $k_1\geq 2.$ Let $H_1=\langle v^{(p_1)}_1\rangle,\cdots, H_s=\langle v^{(p_1)}_s\rangle $ be the complete list of distinct cyclic subgroups of order $p_1$ of $G.$ Now $k_1\geq 2$ implies that $s\geq 3.$ Therefore we can choose three distinct vertices  $v_{i_1}^{(p_1)}, v_{i_2}^{(p_1)}, v_{i_3}^{(p_1)}$ (order of each vertex is $p_1$) from three distinct cyclic subgroups $H_{i_1}, H_{i_2}, H_{i_3}$ respectively. Now we show that $e, v_{i_1}^{(p_1)}, v_{i_2}^{(p_1)}, v_{i_3}^{(p_1)}$ form an induced subgraph isomorphic to the graph $\Gamma_{1, 3}(e, v_{i_1}^{(p_1)}, v_{i_2}^{(p_1)}, v_{i_3}^{(p_1)}).$ From the definition of the power graph $e$ is edge connected with $v_{i_1}^{(p_1)}, v_{i_2}^{(p_1)} \text{ and }  v_{i_3}^{(p_1)}.$ Now $v_{i_1}^{(p_1)}, v_{i_2}^{(p_1)} \text{ and }  v_{i_3}^{(p_1)}$ are the generators of three distinct cylcic subroups $H_{i_1}, H_{i_1}$ and $H_{i_1}$ respectively. Also order of each cyclic subgroup is $p_1.$ As a result, $v_{i_j}^{(p_1)}\nsim v_{i_k}^{(p_1)},$ for each $i, j\in\{1, 2, 3\} \text{ and } i\neq j.$ Hence the theorem.  
\end{proof}
Here we focus on the finite group $G$ such that $G$ is non abelian nilpotent. In this case, we have the following:
\begin{theorem}\label{thm: P(G) line graph, G non abelian nilpotent}
Let $G$ be a non abelian nilpotent group. Then there does not exists any graph $\Gamma$  such that $\mathcal{P}(G)=L(\Gamma).$   	
\end{theorem}
\begin{proof}
It is given that $G$ is nilpotent. So, by Lemma \ref{Nilpotent group charecterization thm}, $G\cong P_1\times \cdots\times P_r,$ where each $P_i$ is a Sylow subgroup of order $p_i^{\alpha_i} \text{ and }\alpha_i\in \mathbb{N}.$ Now divide the prove of this theorem in several cases.

Case 1: First let $r\geq 3.$ In this case $|G|$ has at least three distinct prime divisors $p_1, p_2$ and $ p_3$ (say). Then by Lemma \ref{m|G, nilpotent G has subgroup of ordder m}, $G$ has three elements $v^{(p_1)}, v^{(p_2)}, v^{(p_3)}$ such that $\text{o}(v^{(p_1)})=p_1, \text{o}(v^{(p_2)})=p_2$ and $\text{o}(v^{(p_3)})=p_3.$	Now from the definition of the power graph, it is clear that the vertices $e, v^{(p_1)}, v^{(p_2)}, v^{(p_3)}$ form an induced subgraph $\Gamma_{1, 3}(e, v^{(p_1)}, v^{(p_2)}, v^{(p_3)})$ in $\mathcal{P}(G).$ 
	
Case 2: Let $r=2,$ then $G\cong P_1\times P_2.$ Now it is given that $G$ is non abelian. Therefore, either $P_1$ or $P_2$ is non-cyclic. Without loss of generality we assume that $P_1$ is non cyclic. Now $P_1$ is non cyclic group of order $p_1^{\alpha},$ for some $\alpha\geq 2, (\alpha=1 \text{ implies that } P_1 \text{ is cyclic}).$ As a result, $P_1$ has at least two elements $v, v'$ such that $v\nsim v'$ in $\mathcal{P}(G).$ Otherwise by Lemma \ref{thm:P(G) compltele iff G cylcic p group} $P_1$ would be a cyclic group. Clearly, $\text{o}(v)$ and $\text{o}(v')$ are power of the prime $p_1.$ Let $v''$ be an element of order $p_2$ in $P_2.$ Now from the definition of the power graph  $v''$ is edge connected to neither $v$ nor $v'.$ Also $e$ is edge connected to the vertices $v, v', v''.$ Hence $\Gamma_{1, 3}(e, v, v', v'')$ is an induced subgraph of $\mathcal{P}(G).$
	
Case 3: Let $r=1.$ In this case, $G$ is a non abelian $p$-group. To prove this case we first prove a claim. 

Claim: Let $H_1=\langle v_1^{(p)}\rangle, \cdots, H_{\ell}=\langle v_{\ell}^{(p)}\rangle$ be the collection of all distinct cyclic subgroups of order $p$ of $G.$ We prove that either $\ell=1$ or $\ell\geq 3.$ So, we have to prove that $\ell$ can not be $2.$ So suppose $\ell=2.$ Since $G$ is $p$-group, the center of $G$ is $Z(G)$ is non trivial. Therefore, $Z(G)$ has a subgroup $H_{i_1}=\langle v_{i_1}^{(p)}\rangle$ (say) of order $p.$ Let $H_{i_2}=\langle v_{i_2}^{(p)}\rangle$ be another cyclic subgroup of order $p,$ (as $\ell\geq 2, \text{ we can choose more than one subgroup of order } p).$ Then $v_{i_1}^{(p)}v_{i_2}^{(p)}=v_{i_2}^{(p)}v_{i_1}^{(p)}$ and $\text{o}(v_{i_1}^{(p)}v_{i_2}^{(p)})=p.$ Take $H_{i_3}=\langle v_{i_1}^{(p)}v_{i_2}^{(p)}\rangle$ and it is easy to see that $H_1\neq H_3$ and $H_2\neq H_3.$ Hence $\ell>2.$
	
Now using the claim we finish the proof of this case. First suppose that $\ell\geq 3.$ Then we can choose three $p$-ordered elements $v_{i_1}^{(p)}, v_{i_2}^{(p)} \text{ and } v_{i_3}^{(p)}$ from $H_{i_1}, H_{i_2}$ and $H_{i_3}$ respectively. Clearly they are not adjacent to each other. Therefore, $\mathcal{P}(G)$ has an induced subgraph $\Gamma_{1, 3}(e, v_{i_1}^{(p)}, v_{i_2}^{(p)}, v_{i_3}^{(p)}).$ Now let $\ell=1.$ i.e., $G$ has exactly one cyclic subgroup of order $p.$ Now by Lemma \ref{p group unique subgrp of order p, g is cyclic}, $p$ can not be an odd prime. Therefore $p=2.$ Then, $G\cong Q_{2^n},$ where $ Q_{2^n}$ is the generalized quaternion group of order $2^n, n\geq 3.$ Now $Q_{2^n}$ has exactly one subgroup $H=\langle v^{(2^{n-1})}\rangle$ of order $2^{n-1}.$ Also, the number of $4$-ordered elements in $Q_{2^n}\setminus H$ is $2^{n-1}.$ Now, the number of distinct $4$-ordered cyclic subgroups in $Q_{2^n}\setminus H$ is $2^{n-2}.$ Let these are $H_1=\langle v_1^{(4)}\rangle, \cdots, H_{2^{n-2}}=\langle v_{2^{n-2}}^{(4)}\rangle.$ Clearly $2^{n-2}\geq 2.$ So, we can choose two vertices $v_{i_1}^{(4)}$ and $v_{i_2}^{(4)}$
from $H_{i_1}$ and $H_{i_2}$ respectively, where $i_1\neq i_2$ and $\text{o}(v_{i_1}^{(4)})=4=\text{o}(v_{i_2}^{(4)}).$ Now $v_{i_1}^{(4)}, v_{i_2}^{(4)}\in Q_{2^n}\setminus H$ implies that  $v^{(2^{n-1})}$ neither edge connected to $v_{i_1}^{(4)}$ nor to $ v_{i_2}^{(4)}.$ Moreover, $v_{i_1}^{(4)}\nsim v_{i_2}^{(4)}$ in $\mathcal{P}(G).$ Therefore, $\mathcal{P}(G)$ has an induced subgraph $\Gamma_{1, 3}(e, v^{(2^{n-1})}, v_{i_1}^{(4)}, v_{i_2}^{(4)}).$ This completes the proof.
\end{proof}
\section{Proof of the main theorem for proper power graphs}
In this portion we give the attention for the proof of Theorem \ref{Thm:P**(G) is line graph for nilpotent group}. Here we first characterize all cyclic groups, for which $\mathcal{P}^{**}(G)=L(\Gamma).$ Then we do same thing for the cases non cyclic abelian and non abelian nilpotent groups.  
\begin{theorem}\label{classify: G cyclic line graph, P^{**}(G)}
Let $G$ be a finite cyclic group. Then there exists a graph $\Gamma$ such that $\mathcal{P}^{**}(G)=L(\Gamma)$  if and only if $G$ is one of the following:
\begin{enumerate}
\item[(a)] 
$G\cong\mathbb{Z}_{p^t}$
\item[(b)]
$G\cong \mathbb{Z}_{pq},$ 	
\end{enumerate}
where $p, q$ are distinct primes and $t\geq 1.$	
\end{theorem}
To prove this theorem first we need to prove the following propositions: 
\begin{proposition}\label{prop:P^{**}(G), does not contain star graph Gamma(1, 3), G cylic}
Let $G$ be a finite cyclic group. Then $\mathcal{P}^{**}(G)$ does not contain the star graph $\Gamma_{1, 3}$ as an induced subgraph if and only if $G$ is one of the following:
\begin{enumerate}
\item[(a)]
$G\cong \mathbb{Z}_{p^t}$
\item[(b)]
$G\cong \mathbb{Z}_{pqr}$	
\item[(c)]
$G\cong \mathbb{Z}_{p^2q^2}$
\item[(d)]
$G\cong \mathbb{Z}_{p^tq},$ 
\end{enumerate}
where $p, q, r$ are primes such that $p\neq q\neq r \text{ and } t\geq 1.$	
\end{proposition}
\begin{proof}
We divide the proof of this proposition in several cases.

Case 1: Let $|G|$ has at least four distinct prime divisors say $p, q, r, p'.$ Since $G$ is cyclic, then $G$ has an element $v^{(pqr)}$ such that $\text{o}(v^{(pqr)})=pqr.$ Also the cyclic subgroup $\langle v^{(pqr)}\rangle$ has elements $v^{(p)}, v^{(q)}, v^{(r)}$ of order  $p, q, r$ respectively. Clearly, the vertices $v^{(pqr)}, v^{(p)}, v^{(q)} \text{ and } v^{(r)}$ are not the generators of the group $G.$ Therefore, by Lemma \ref{lemma: classification of dominating vertices of power graph} $v^{(pqr)}, v^{(p)}, v^{(q)} \text{ and } v^{(r)}\in V(\mathcal{P}^{**}(G)).$ Now $v^{(p)}, v^{(q)} \text{ and } v^{(r)}\in \langle v^{(pqr)}\rangle$ implies that $ v^{(pqr)}$ is edge connected to the vertices $v^{(p)}, v^{(q)} \text{ and } v^{(r)}.$ Again, $p, q, r$ are three distinct primes, therefore from the definition of power graph the vertices $v^{(p)}, v^{(q)} \text{ and } v^{(r)}$ are not adjacent to each other. As a result, $\mathcal{P}^{**}(G)$ has an induced subgraph $\Gamma_{1, 3}(v^{(pqr)}, v^{(p)}, v^{(q)}, v^{(r)}).$ 

Case 2: Let $|G|=p^{\alpha}q^{\beta}p_3^{\gamma},$ where at least one of $\alpha, \beta, \gamma\geq 2.$ Then applying the same argument as in the Case 1, we can conclude that $\mathcal{P}^{**}(G)$ has an induced subgraph $\Gamma_{1, 3}(v^{(pqr)}, v^{(p)}, v^{(q)}, v^{(r)}).$ 
 
Now we show that if $G\cong\mathbb{Z}_{pqr},$ then $\mathcal{P}^{**}(\mathbb{Z}_{pqr})$ does not contain any induced subgraph isomorphic to $\Gamma_{1, 3}.$ Note that by Lemma \ref{lemma: classification of dominating vertices of power graph}, the identity and all the generators of the group $\mathbb{Z}_{pqr}$ are the complete list of domminating vertices of the graph $\mathcal{P}(\mathbb{Z}_{pqr}).$ Therefore, the identity and all the generators of the group $\mathbb{Z}_{pqr}$ are not in the vertex set $V(\mathcal{P}^{**}(\mathbb{Z}_{pqr})).$ If possible $\mathcal{P}^{**}(\mathbb{Z}_{pqr})$ contains an induced subgraph $\Gamma_{1, 3}(v, v_1, v_2, v_3),$ for some vertices $v, v_1, v_2, v_3\in V(\mathcal{P}^{**}(\mathbb{Z}_{pqr})).$ So, $v\sim v_i $ for each $i\in \{1, 2, 3\}$  and $v_i\nsim v_j$ for each pair $i, j(i\neq j)\in \{1, 2, 3\}.$ Now the possible order of the vertices of the graph $\mathcal{P}^{**}(\mathbb{Z}_{pqr})$ are either $p$ or $q$ or $r$ or $pq$ or $pr$ or $qr.$ Now order of $v$ could either $p$ or $q$ or $r$ or $pq$ or $qr$ or $pr.$ First suppose that $\text{o}(v)=pq.$ Then it is cleared (from the definition of power graph) that $v$ is edge connected only with each $pq$-ordered, $p$-ordered and $q$-ordered vertices in $\mathcal{P}^{**}(\mathbb{Z}_{pqr}).$ Let $V(t)$ be the collection of all vertices of order $t.$ So we have to choose $v_1, v_2, v_3$ from $V(p)\cup V(q)\cup V({pq})$ such that no two of them are adjacent. Note that we can not choose more than one vertex from any one of the set $V(p), V(q), V({pq}).$ In fact, all the vertices in any one of the set $V(p), V(q), V({pq})$ form a clique. So, without loss of generality we assume that $v_1\in V(p), v_2\in V(q)$ and $v_3\in V(pq).$ But it is cleared that $v_3$ is edge connected to both of the vertices $v_1$ and $v_2.$ This violates the condition $v_i\nsim v_j$ for each pair $i, j\in \{1, 2, 3\}.$  Similarly, we can show that the graph $\mathcal{P}^{**}(\mathbb{Z}_{pqr})$ don't have any induced subhgraph isomorphic to $\Gamma_{1, 3}(v, v_1, v_2, v_3)$ for the other possible choices of the $\text{o}(v), \text{o}(v_1), \text{o}(v_2), \text{o}(v_3).$
	
Case 3: Let $|G|$ has two distinct prime divisors $p$ and $q.$ Suppose $G\cong\mathbb{Z}_{p^{t_1}}\times \mathbb{Z}_{q^{t_2}}.$  First we consider that $t_1\geq 3 \text{ and }t_2\geq 2,$ (if $t_1\geq 2 \text{ and }t_2\geq 3,$ then also similar result holds). In this case, $\mathcal{P}^{**}(\mathbb{Z}_{p^{t_1}}\times \mathbb{Z}_{q^{t_2}})$ have an induced subgraph $\Gamma_{1, 3}(v^{(p)}, v^{(p^3)}, v^{(p^2q)}, v^{(pq^2)}),$ where $\text{o}(v^{(p)})=p, \text{o}( v^{(p^3)})=p^3, \text{o}(v^{(p^2q)})=p^2q, \text{o}(v^{(pq^2)})=pq^2.$ 
	
If $G\cong\mathbb{Z}_{p^{t}q},$ where $t\geq 1.$ If we proceed exactly same way as in the proof of the case $G\cong\mathbb{Z}_{pqr},$ we can conclude that $\mathcal{P}^{**}(\mathbb{Z}_{p^{t}q})$ does not have any induced subgraph isomorphic to star graph $\Gamma_{1, 3}.$	
	
If $G\cong\mathbb{Z}_{p^2q^2}.$ As it was done in the case where $G\cong\mathbb{Z}_{pqr}$ we can prove that $\mathcal{P}^{**}(\mathbb{Z}_{p^2q^2})$ does not have any induced subgraph isomorphic to $\Gamma_{1, 3}.$

Case 4: Let $|G|$ has exactly one prime divisor $p$ say. Now $G\cong\mathbb{Z}_{p^t}$ as a result  $\mathcal{P}^{**}(\mathbb{Z}_{p^t})$ is complete. Hence the result.
\end{proof}
\begin{proposition}\label{prop: P^{**}(G), does not contain Gamma_2, G cyclic}
Let $G$ be a cyclic group. Then $\mathcal{P}^{**}(G)$ does not contain the graph $\Gamma_2$ (a graph in Figure \ref{fig:line grapph theory}) as an induced subgraph if and only if $G$ is one of the following:
\begin{enumerate}
\item[(a)] 
$G\cong \mathbb{Z}_{p^t}$
		\item[(b)] 
		$G\cong\mathbb{Z}_{pq}$
		\item[(c)] 
		$G\cong \mathbb{Z}_{12}$
		\item[(d)] 
		$G\cong\mathbb{Z}_{18}.$	
	\end{enumerate}	
\end{proposition}
\begin{proof}
Let $G$ be a cyclic group of order $n$ such that $n$ has at least three distinct prime divisors $p<q<r$ (say). Let $v_1^{(qr)}, v_2^{(qr)}, v_3^{(qr)}, v^{(q)}, v^{(r)}\in G$ such that   $\text{o}(v_1^{(qr)})=\text{o}(v_2^{(qr)})=\text{o}(v_3^{(qr)})=qr, \text{o}(v^{(q)})=q, \text{o}(v^{(r)})=r$ (since $G$ is cylcic and $|G|$ has at least three distinct prime divisors $p, q, r$ with $p<q<r,$ then clearly $\phi(qr)\geq 3).$ Clearly, $v_1^{(qr)}, v_2^{(qr)}, v_3^{(qr)},v^{(q)}, v^{(r)}\in V(\mathcal{P}^{**}(G)).$ Now we replace the vertices of the graph $\Gamma_2$ in Figure \ref{fig:line grapph theory} in the following way:
\[6 \text{  by } v_1^{(qr)}, 8 \text{  by } v_2^{(qr)}, 10 \text{  by } v_3^{(qr)}, 7 \text{  by } v^{(q)}, 9 \text{  by } v^{(r)}.\] Clearly, the resulting induced graph isomorphic to $\Gamma_2.$
	
If $G\cong\mathbb{Z}_{p^tq}, t\geq 3,$ in this case $\mathcal{P}^{**}(G)$ have the vertices namely, $v_1^{(p^2q)}, v_2^{(p^2q)}, v_3^{(p^2q)}, v^{(p^2)}, v^{(q)},$ where $\text{o}(v_1^{(p^2q)})=\text{o}(v_1^{(p^2q)})=\text{o}(v_1^{(p^2q)})=p^2q, (\text{ as } \phi(p^2q)\geq 3) \text{o}(v^{(p^2)})=p^2, \text{o}(v^{(q)})=q.$ Now we replace the vertices of the graph $\Gamma_2$ in the following way: 
\[6 \text{  by } v_1^{(p^2q)}, 8 \text{  by }v_2^{(p^2q)}, 10 \text{  by } v_3^{(p^2q)}, 7 \text{  by } v^{(p^2)}, 9 \text{  by } v^{(q)}.\] Clearly, the resulting induced graph is isomorphic to $\Gamma_2$ (in Figure \ref{fig:line grapph theory}).
	
Let $G\cong\mathbb{Z}_{p^2q}.$ Suppose $\phi(pq)\geq 3,$ then $\mathcal{P}^{**}(G)$ has three vertices $v_1^{(pq)}, v_2^{(pq)}, v_3^{(pq)}$ such that $\text{o}(v_1^{(pq)})=pq=\text{o}(v_2^{(pq)})=\text{o}(v_3^{(pq)}).$ Also it has two vertices $v^{(p)}$ and $v^{(q)}$ such that $\text{o}(v^{(p)})=p$ and $\text{o}(v^{(q)})=q.$ Now we replace the vertices of the graph $\Gamma_2$ in the following way: 
\[6 \text{  by } v_1^{(pq)}, 8 \text{  by } v_2^{(pq)}, 10 \text{  by } v_3^{(pq)}, 7 \text{  by } v^{(p)}, 9 \text{  by } v^{(q)}.\] Clearly, the resulting induced graph is isomorphic to $\Gamma_2$ (in Figure \ref{fig:line grapph theory}).

Now $\phi(pq)=2$ if and only if either $p=2, q=3$ or $q=2, p=3.$ So, either $G\cong\mathbb{Z}_{12}$ or $G\cong\mathbb{Z}_{18}.$ In these two cases we prove that $\mathcal{P}^{**}(\mathbb{Z}_{12})$ and $\mathcal{P}^{**}(\mathbb{Z}_{18})$ don't have an induced subgraph isomorphic to the graph $\Gamma_2$ in Figure \ref{fig:line grapph theory}. In fact, in $\mathcal{P}^{**}(\mathbb{Z}_{12})$ there are exactly two vertices namely, $v_1^{(3)}, v_2^{(3)}$ of order $3$ and they generate same cyclic group. Also in $\mathcal{P}^{**}(G), \text{deg}(v_1^{(3)})=3=\text{deg}(v_2^{(3)})$ and $v_1^{(3)}, v_2^{(3)}$ are the only vertices of degree $3.$ But the graph $\Gamma_2$ has two vertices of degree $3$ and they are not edge connected. Therefore, $\mathcal{P}^{**}(\mathbb{Z}_{12})$ has no induced subgraph isomorphic to $\Gamma_2.$ Again, in $\mathcal{P}^{**}(\mathbb{Z}_{18})$ there is no vertex $v$ such that $\text{deg}(v)=3.$ Hence in both of the cases it is not possible. Now if $G\cong\mathbb{Z}_{p^t},$ then $\mathcal{P}^{**}(\mathbb{Z}_{p^t})$ is complete. So $\mathcal{P}^{**}(\mathbb{Z}_{p^t})$ does not have any induced subgraph isomorphic to $\Gamma_2.$ Also for the group $G\cong \mathbb{Z}_{pq},$ the graph $\mathcal{P}^{**}(\mathbb{Z}_{pq})$ is disjoint union of two cliques. Therefore, $\mathcal{P}^{**}(\mathbb{Z}_{pq})$ does not have any induced subgraph isomorphic to the graph $\Gamma_2$ (in Figure \ref{fig:line grapph theory}). This completes the proposition. 
\end{proof}
\begin{proof}[ Proof of Theorem \ref{classify: G cyclic line graph, P^{**}(G)}]
Clearly, from Propositions \ref{prop:P^{**}(G), does not contain star graph Gamma(1, 3), G cylic} and \ref{prop: P^{**}(G), does not contain Gamma_2, G cyclic}, it follows that $\mathcal{P}^{**}(G)$ is a line graph of some graph $\Gamma$ if and only if either $G\cong\mathbb{Z}_{p^t}$ or $G\cong\mathbb{Z}_{pq}$ or $G\cong\mathbb{Z}_{12}$ or $G\cong\mathbb{Z}_{18}.$ Now we show that if $G\cong\mathbb{Z}_{12}$ or $\mathbb{Z}_{18},$ then there exists no graph $\Gamma$ such that $\mathcal{P}^{**}(G)$ is line graph of $\Gamma.$ Let  $G\cong\mathbb{Z}_{12},$ then we show that $\mathcal{P}^{**}(\mathbb{Z}_{12})$ has an induced subgraph isomorphic to the graph $\Gamma_4$ in Figure \ref{fig:line grapph theory}. Clearly $\mathbb{Z}_{12}$ has elements $v_1^{(6)}, v_2^{(6)}, v_1^{(3)}, v_2^{(3)}, v^{(2)}, v^{(4)}$ such that $\text{o}(v_1^{(6)})=\text{o}(v_2^{(6)})=6, \text{o}(v_1^{(3)})=\text{o}(v_2^{(3)})=3, \text{o}(v^{(2)})=2 \text{ and }\text{o}(v^{(4)})=4.$ Now it is easy to see that $\mathcal{P}^{**}(\mathbb{Z}_{12})$ contains the graph $\Gamma_4$ as an induced subgraph by replacing the vertices of the graph $\Gamma_4$ in the following way:
\[17 \text{ by } v_1^{(6)}, 18 \text{ by } v^{(2)}, 19 \text{ by } v_2^{(6)}, 20 \text{ by } v_1^{(3)}, 21 \text{ by } v^{(4)}, \text{ and } 22 \text{ by } v_2^{(3)}.\]  
Let $G\cong\mathbb{Z}_{18},$ then we show that  $\mathcal{P}^{**}(G)$ has an induced subgraph isomorphic to the graph $\Gamma_3$ in Figure \ref{fig:line grapph theory}. The group $\mathbb{Z}_{18}$ has elements $v_1^{(3)}, v_2^{(3)}, v_1^{(6)}, v_2^{(6)}, v_1^{(9)}, v_2^{(9)}$ such that $\text{o}(v_1^{(3)})=\text{o}(v_2^{(3)})=3, \text{o}(v_1^{(6)})=\text{o}(v_2^{(6)})=6 \text{ and } \text{o}(v_1^{(9)})= \text{o}(v_2^{(9)})=9.$ Now we replace the vertices of the graph $\Gamma_3$ in the following way:
\[11 \text{ by } v_1^{(3)}, 12 \text{ by } v_1^{(6)}, 13 \text{ by } v_2^{(3)}, 14 \text{ by } v_1^{(9)}, 15 \text{ by } v_2^{(6)}, \text{ and } 16 \text{ by } v_2^{(9)}.\] Then the resulting graph is isomorphic to the graph $\Gamma_3.$ Hence the graphs $\mathcal{P}^{**}(\mathbb{Z}_{12})$ and $\mathcal{P}^{**}(\mathbb{Z}_{18})$ are not line graph.  
	
Let $G\cong\mathbb{Z}_{pq},$ in this case $\mathcal{P}^{**}(\mathbb{Z}_{pq})$ is the line graph of the graph $\Gamma_{1, \phi(p)}\bigoplus \Gamma_{1, \phi(q)}.$ Also for the cyclic $p$-group $\mathbb{Z}_{p^r}, \mathcal{P}^{**}(\mathbb{Z}_{p^r})$ is the empty graph (empty graph is line graph). Hence the theorem.	
\end{proof}
Now we want to describe all non cyclic abelian groups $G$ for which $\mathcal{P}^{**}(G)$ is a line graph. For this case we have the following: 
\begin{theorem}\label{thm: non cyclic abeln grp line graph of {P}^{**}(G)}
Let $G$ be a non cyclic abelian group. Then $\mathcal{P}^{**}(G)$ is a line graph of some graph $\Gamma$ if and only if $G$ is one of the following:
\begin{enumerate}
\item[(a)] 
$G\cong \mathbb{Z}_2\times \mathbb{Z}_{2^{2}}$
\item[(b)]	
$G\cong \mathbb{Z}_{2^2}\times \mathbb{Z}_{2^2}$
\item[(c)]
$G\cong\mathbb{Z}_p\times \cdots\times \mathbb{Z}_p, p \text{ is prime }.$	
\end{enumerate}	
\end{theorem}
To prove Theorem \ref{thm: non cyclic abeln grp line graph of {P}^{**}(G)}, we use Lemma \ref{line graph}. According to this lemma we have to characterize all non cyclic abelian groups for which  $\mathcal{P}^{**}(G)$ has an induced subgraph isomorphic to the star graph $\Gamma_{1, 3}.$ Now Proposition \ref{Prop: noncyclic, abln, {P}^{**}(G) does not contain Gamma_{1, 3}  } completely describes this case.
\begin{proposition}\label{Prop: noncyclic, abln, {P}^{**}(G) does not contain Gamma_{1, 3}  }
Let $G$ be a non cyclic abelian group. Then $\mathcal{P}^{**}(G)$ does not contain $\Gamma_{1, 3}$ as an induced subgraph if and only if $G$ is one of the following:
\begin{enumerate}
\item[(a)] 
$G\cong \mathbb{Z}_2\times \mathbb{Z}_{2^{2}}$
\item[(b)]
$G\cong \mathbb{Z}_{2^2}\times \mathbb{Z}_{2^2}$
\item[(c)]
$G\cong\mathbb{Z}_p\times \cdots\times \mathbb{Z}_p, p \text{ is prime }.$	
\end{enumerate}	
\end{proposition}
\begin{proof}
First let $|G|$ has at least two distinct prime divisors. In this case, we show that  $\mathcal{P}^{**}(G)$ has an induced subgraph isomorphic to $\Gamma_{1, 3}(v, v', v'', v'''),$ for some vertices $v, v', v'', v'''\in V(\mathcal{P}^{**}(G)).$ It is given that $G$ is non-cyclic abelian group. Therefore,
\[G\cong\mathbb{Z}_{p^{t_{11}}_1}\times \cdots\times\mathbb{Z}_{p^{t_{1k_1}}_1}\times \mathbb{Z}_{p^{t_{21}}_2}\times\cdots\times\mathbb{Z}_{p^{t_{2k_2}}_2}\times\cdots\times
\mathbb{Z}_{p^{t_{r1}}_r}\times\cdots\times\mathbb{Z}_{p^{t_{rk_r}}_r},\] where $1\leq t_{i1}\leq t_{i2}\leq\cdots\leq t_{ik_i}, $ for all $i\in [r], r\geq 2, k_i\geq1$ and there exists at least one $k_i$ such that $k_i\geq 2.$ So, without loss of generality, we assume that $k_1\geq 2.$ Consider \[V=\{ (\underbrace{\bar{a}, \bar{b},\bar{0}, \cdots, \bar{0}}_{k_1 \text{ times }}, \bar{c}, \bar{0}, \cdots, \bar{0}): \bar{a}, \bar{b}, \bar{c}\in G \text{ and } \text{o}(\bar{a})=\text{o}(\bar{b})=p_1, \text{o}(\bar{c})=p_2\}.\] Clearly, $V$ is a subset of $G$ and each element of $V$ is of order $p_1p_2$ and $|V|=(p_1^2-1)(p_2-1).$ Again these $(p_1^2-1)(p_2-1)$ number of elements form $\frac{(p_1^2-1)(p_2-1)}{\phi(p_1p_2)}=p_1+1$ number of distinct cyclic groups say $H_1, H_2, \cdots, H_{p_1+1},$ where order of each $H_i$ is $p_1p_2.$ Now, it is easy to see that, the cyclic group $\langle \underbrace{(\bar{0}, \cdots, \bar{0}}_{k_1\text{ times }}, \bar{c}, \bar{0}\cdots, \bar{0} )\rangle$ is contained in each of the cyclic groups $H_1, \cdots, H_{p_1+1},$ where $\text{o}(\bar{c})=p_2.$ Since $p_1+1\geq 3,$ we can choose three distinct vertices $v_{i_1}^{(p_1p_2)}, v_{i_2}^{(p_1p_2)} \text{ and } v_{i_3}^{(p_1p_2)}$ from $H_{i_1}, H_{i_2}$ and $H_{i_3}$ (respectively) such that $\text{o}(v_{i_j}^{(p_1p_2)})=p_1p_2,$ for $j\in \{1, 2, 3\}$ and $i_1, i_2, i_3\in \{1, \cdots, p_1+1\}.$ Also we take the vertex $v^{(p_2)}= \underbrace{(\bar{0}, \cdots, \bar{0}}_{k_1\text{ times }}, \bar{c}, \bar{0}\cdots, \bar{0} ).$ Then we get $\Gamma_{1, 3}(v, v', v'', v''')$ as an induced subgraph in the graph $\mathcal{P}^{**}(G),$ where $v=v^{(p_2)},v'= v_{i_1}^{(p_1p_2)}, v''=v_{i_2}^{(p_1p_2)}, v'''=v_{i_3}^{(p_1p_2)}.$
 
Now suppose that $G$ is a non cyclic abelian $p$-group. Then we can say that $G\cong \underbrace{\mathbb{Z}_p\times\cdots\times \mathbb{Z}_p}_{k (\geq 0) \text{ times }}\times \mathbb{Z}_{p^{t_1}}\times\cdots \times\mathbb{Z}_{p^{t_r}},$ where $t_1\leq t_2\leq\cdots\leq t_r, t_i\geq 2 \text{ for all }i.$ For this particular groups ( non cyclic abelian $p$-groups), we break the prove in several cases.

Case 1: First suppose that $r\geq 3.$ In this case, we show that $\mathcal{P}^{**}(G)$ has an induced subgraph $\Gamma_{1, 3}(v^{(p)}, v_1^{(p^{(t_1)})}, v_2^{(p^{(t_1)})}, v_3^{(p^{(t_1)})}),$ where 
\begin{align*}
&v^{(p)}=(\underbrace{\bar{0}, \cdots, \bar{0}}_{k \text{ times }}, \bar{a}, \bar{0}, \bar{0}, \bar{0}, \cdots, \bar{0})\\
&v_1^{(p^{(t_1)})}=(\underbrace{\bar{0}, \cdots, \bar{0}}_{k \text{ times }}, \bar{1}, \bar{b}, \bar{c}, \bar{0},\cdots, \bar{0})  \\
&v_2^{(p^{(t_1)})}=(\underbrace{\bar{0}, \cdots, \bar{0}}_{ k \text{ times }}, \bar{1}, \bar{b}, \bar{0}, \bar{0}, \cdots, \bar{0})\\ 
&v_3^{(p^{(t_1)})}=(\underbrace{\bar{0}, \cdots, \bar{0}}_{k \text{ times }}, \bar{1}, \bar{0}, \bar{0}, \bar{0}, \cdots, \bar{0}),
\end{align*}
${a}=p^{t_1-1}$ and $\text{o}(\bar{b})=\text{o}(\bar{c})=p.$ Clearly, $\text{o}(\bar{a})=p.$ Again $\bar{1}$ is a generator of the group $\mathbb{Z}_{p^{t_1}}$ and we can write $\bar{a}=p^{t_1-1}\bar{1}.$ Also $t_1\geq 2$ implies that $t_1-1\geq1.$ Now $\text{o}(\bar{b})=\text{o}(\bar{c})=p$ implies that $p^{t_1-1}v_1^{(p^{(t_1)})}=v^{(p)}.$ Therefore, $v_1^{(p^{(t_1)})}\sim v^{(p)}.$ Similarly, we can show that $v_2^{(p^{(t_1)})}\sim v^{(p)}$ and $v_3^{(p^{(t_1)})}\sim v^{(p)}.$ Now we show that $v_i^{(p^{(t_1)})}\nsim v_j^{(p^{(t_1)})},$ for all $i, j\in \{1, 2, 3\}.$ Note that $\text{o}(v_i^{(p^{(t_1)})})=p^{t_1}$ for each $i.$ So $v_i^{(p^{(t_1)})}\sim v_j^{(p^{(t_1)})}$ if and only if $\langle v_i^{(p^{(t_1)})}\rangle=\langle v_j^{(p^{(t_1)})}\rangle.$ But clearly it is not possible from the construction of the vertices.
 
Case 2: Let $r=2,$ so $G\cong\underbrace{{\mathbb{Z}_p\times\cdots \mathbb{Z}_p}}_{k (\geq 0) \text{ times }}\times \mathbb{Z}_{p^{t_1}}\times \mathbb{Z}_{p^{t_2}},$ where $t_1\leq t_2, t_i\geq 2 \text{ for all }i.$ 

Subcase 1: First suppose that $p$ is an odd prime. Since $\phi(p)\geq 2,$ we can choose two distinct elements $\bar{b}$ and $\bar{c}$ from $\mathbb{Z}_{p^{t_2}}$ such that $\text{o}(\bar{b})=\text{o}(\bar{c})=p.$ Also we can take another $p$-ordered element $\bar{a}\in \mathbb{Z}_{p^{t_1}}$ such that $a=p^{t_1-1}.$ Now we consider the vertices 
\begin{align*}
&v^{(p)}=(\bar{0}, \cdots, \bar{0}, \bar{a}, \bar{0}),&
&v^{(p^{t_1})}_1=(\bar{0}, \cdots, \bar{0}, \bar{1}, \bar{b}), \\
&v^{(p^{t_1})}_2=(\bar{0}, \cdots, \bar{0}, \bar{1}, \bar{c}),&  
&v^{(p^{t_1})}_3=(\bar{0}, \cdots, \bar{0}, \bar{1}, \bar{0}).
\end{align*}
Continuing as Case 1, we can show that  $\mathcal{P}^{**}(G)$ has an induced subgraph $\Gamma_{1, 3}(v^{(p)}, v^{(p^{t_1})}_1, v^{(p^{t_1})}_2, v^{(p^{t_1})}_3).$   

Subcase 2: Here we focus on the case $p=2.$ So,  $G\cong\underbrace{\mathbb{Z}_{2}\times\cdots\times \mathbb{Z}_{2}}_{k (\geq 0)\text{ times }}\times\mathbb{Z}_{2^{t_1}}\times\mathbb{Z}_{2^{t_2}}.$ In this case we first consider that at least one $t_1$ and $t_2\geq 3.$  Without loss of generality we assume that $t_1\geq 3.$ Here we consider the vertices 
\begin{align*}
&v=(\bar{0}, \cdots, \bar{0}, \bar{a}, \bar{0}),& &v_1=(\bar{0}, \cdots, \bar{0}, \bar{1}, \bar{b}),\\
&v_2=(\bar{0}, \cdots, \bar{0}, \bar{1}, \bar{c}), & 
&v_3=(\bar{0}, \cdots, \bar{0}, \bar{1}, \bar{0}),
\end{align*}
where $a=2^{t_1-1}, \text{o}(\bar{b})=2, \text{o}(\bar{c})=4.$ Clearly, $\mathcal{P}^{**}(G)$ has an induced subgrap $\Gamma_{1, 3}(v^{(2)}, v_1^{(2^{t_1})}, v_2^{(t_1)}, v_3^{(t_1)}).$  

Let $G\cong \underbrace{\mathbb{Z}_2\times \cdots\times\mathbb{Z}_2}_{k\geq 0 \text{ times }}\times\mathbb{Z}_{2^2}\times \mathbb{Z}_{2^2}.$ It is easy to see that $\mathcal{P}^{**}(G)$ has an induced subgraph $\Gamma_{1, 3}(v^{(2)}, v_1^{(4)}, v_2^{(4)}, v_3^{(4)}),$ where 
\begin{align*}
&v^{(2)}=(\bar{0}, \cdots, \bar{0}, \bar{0}, \bar{0},\bar{2}),&
&v_1^{(4)}=(\bar{0}, \cdots, \bar{0}, \bar{1}, \bar{0}, \bar{1})\\
&v_2^{(4)}=(\bar{0}, \cdots, \bar{0}, \bar{0}, \bar{0}, \bar{1}),& 
&v_3^{(4)}=(\bar{0}, \cdots, \bar{0}, \bar{0}, \bar{2}, \bar{1})
\end{align*}
Case 3: Let $r=1.$ In this case $G\cong \underbrace{\mathbb{Z}_p\times \cdots \times \mathbb{Z}_p}_{k(\geq 1)\text{ times }}\times \mathbb{Z}_{p^t}, t\geq 2.$ 

Subcase 1: First suppose that $p\geq 3.$ Then $\mathcal{P}^{**}(G)$ has an induced subgraph $\Gamma_{1, 3}(v^{(p^{t-1})}, v_1^{(p^t)}, v_2^{(p^t)}, v_3^{(p^t)}),$  where 
\begin{align*}
&v^{(p^{t-1})}=(\bar{0},\cdots, \bar{0}, \bar{0}, \bar{p}),& 
&v_1^{(p^t)}=(\bar{0}, \cdots, \bar{0}, \bar{2}, \bar{1}),\\ 
&v_2^{(p^t)}=(\bar{0}, \cdots, \bar{0}, \bar{0}, \bar{1}),& 
&v_3^{(p^t)}=(\bar{0}, \cdots, \bar{0},\bar{1}, \bar{1})
\end{align*}
Subcase 2: Let $p=2.$ Then $G\cong \underbrace{\mathbb{Z}_2\times \cdots\mathbb{Z}_2}_{ k(\geq 1)\text{ times}}\times\mathbb{Z}_{2^t}, t\geq 2.$
In this case, first suppose that  $k\geq 2.$ Then the following vertices form an induced subgraph $\Gamma_{1, 3}(v^{(2^{t-1})}, v_1^{(2^t)}, v_2^{(2^t)}, v_3^{(2^t)}),$ where
\begin{align*}
&v^{(2^{t-1})}=(\bar{0}, \cdots, \bar{0}, \bar{0}, \bar{0},\bar{2}),&
&v_1^{(2^t)}=(\bar{0}, \cdots, \bar{0}, \bar{0}, \bar{1}, \bar{1})\\
&v_2^{(2^t)}=(\bar{0}, \cdots, \bar{0}, \bar{0}, \bar{0},\bar{1}),& 
&v_3^{(2^t)}=(\bar{0}, \cdots, \bar{0}, \bar{1}, \bar{0}, \bar{1}).
\end{align*}
 Now let $k=1.$ Then
$G\cong\mathbb{Z}_2\times \mathbb{Z}_{2^t}, t\geq 2.$ In this case, we show that $\mathcal{P}^{**}(G)$ has an induced subgraph isomorphic to $\Gamma_{1, 3}$ if and only if $t\geq 3.$ In fact, for $t\geq 3$ we have the induced subgraph $\Gamma_{1, 3}(v, v_1, v_2, v_3),$ where $v=(\bar{0}, \bar{4}), v_1=(\bar{1}, \bar{2}), v_2=(\bar{0}, \bar{1}), v_3=(\bar{1}, \bar{1}).$ 
Let $G\cong\mathbb{Z}_2\times\mathbb{Z}_{2^2}.$ Clearly, $\mathcal{P}^{**}(G)$ is the graph in Figure \ref{fig:pwr grapph for  group of order Z_2*Z_2^2}
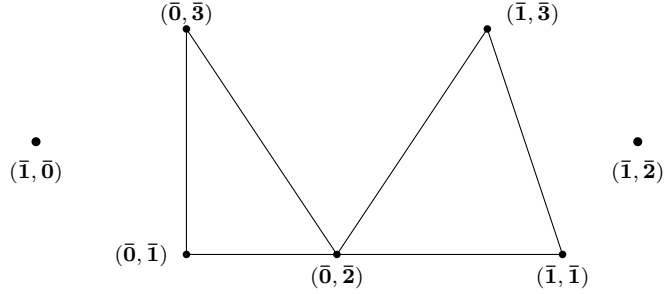
\begin{figure}[H]
	\tiny
	\centering
	\begin{tikzpicture}[scale=1]
	\tikzstyle{edge_style} = [draw=black, line width=2mm, ]
	\draw (1,0)--(3,0);
	\draw (1, 0)--(1,3);
	\draw (1,3)--(3,0);
	\draw (3,0)--(6,0);
	\draw (3,0)--(5,3);
	\draw (5,3)--(6,0);
	\node (e) at (.4,0){$\bf{(\bar{0}, \bar{1})}$};
	\node (e) at (3,-.3){$\bf{(\bar{0}, \bar{2})}$};
	\node (e) at (6,-.3){$\bf{(\bar{1}, \bar{1})}$};
	\node (e) at (1,3.2){$\bf{(\bar{0}, \bar{3})}$};
	\node (e) at (5.6,3.2){$\bf{(\bar{1}, \bar{3})}$};
	\node (e) at (-1,1.1){$\bf{(\bar{1}, \bar{0})}$};
	\node (e) at (7,1.1){$\bf{(\bar{1}, \bar{2})}$};
	\fill[black!100!] (3,0) circle (.05);
	\fill[black!100!] (1,0) circle (.05);
	\fill[black!100!] (1,3) circle (.05);
	\fill[black!100!] (5,3) circle (.05);
	\fill[black!100!] (6,0) circle (.05);
	\fill[black!100!] (-1,1.5) circle (.06);
	\fill[black!100!] (7,1.5) circle (.06);
	\end{tikzpicture}
	\caption{The proper power graph $\mathcal{P}^{**}(Z_2\times Z_2^2)$}
	\label{fig:pwr grapph for  group of order Z_2*Z_2^2}	
\end{figure}
Now we show that there is no vertices $v, v_1, v_2, v_3$ in $V(\mathcal{P}^{**}(G))$ such that $\mathcal{P}^{**}(G)$ has an induced subgraph $\Gamma_{1, 3}(v, v_1, v_2, v_3).$ The graph in Figure  \ref{fig:pwr grapph for  group of order Z_2*Z_2^2} has only one vertex namely, $(\bar{0}, \bar{1})$ such that $\text{deg}(\bar{0}, \bar{1})=4$ and the degree of all other vertices are $2.$ So, to form $\Gamma_{1, 3}(v, v_1, v_2, v_3)$ we should take $v=(\bar{0}, \bar{1}).$ Now it is clear that we can not choose $v_1, v_2, v_3$ such that no two of these three vertices are adjacent.
  
Let $G\cong\mathbb{Z}_{2^2}\times\mathbb{Z}_{2^2},$ then we show that $\mathcal{P}^{**}(G)$ does not have any induced subgraph isomorphic to $\Gamma_{1, 3}.$ If possible there is an induced subgraph $\Gamma_{1, 3}(v, v_1, v_2, v_3)$ for some $v, v_1, v_2, v_3\in \mathcal{P}^{**}(G).$ Then we show that $\text{o}(v)=2.$ First note that $v\sim v_i (i=1, 2, 3)$ implies $\langle v\rangle\subset \langle v_i\rangle,$ for all $i.$ If $\langle v_1\rangle\subset \langle v\rangle,$ then $v\sim v_2$ implies that either $\langle v\rangle\subset \langle v_2\rangle$ or $\langle v_2\rangle\subset \langle v\rangle.$ Since $G$ is a $p$-group, then in any cases $v_1\sim v_2,$ which is not possible. So, $\text{o}(v)=2$ and $\text{o}(v_i)=4,$ for all $i.$ First consider the two ordered element $v=(\bar{2}, \bar{0}).$ Our claim is that $v$ is edge connected with exactly two four ordered elements $v_1, v_2$ (say) such that $\langle v_1\rangle\neq \langle v_2\rangle.$ Clearly, $(\bar{2}, \bar{ 0})\nsim (\bar{a}, \bar{b}),$ where $\bar{b}=4.$ By our condition $\text{o}(\bar{a})=4.$ Therefore, $(\bar{2}, \bar{0})$ is edge connected with the vertices $(\bar{1},\bar{0}), (\bar{3},\bar{0}), (\bar{1},\bar{2}), (\bar{3},\bar{2}).$ Also, $\langle (\bar{1}, \bar{0})\rangle=\langle (\bar{3}, \bar{0})\rangle$ and $\langle (\bar{1}, \bar{2})\rangle=\langle (\bar{3}, \bar{2})\rangle.$ This proves our claim. The same claim holds for the vertex $(\bar{0}, \bar{2}).$ The remaining two ordered element is $(\bar{2}, \bar{2}).$ Let $(\bar{2}, \bar{2})\sim (\bar{a}, \bar{b}).$ Then we show that $\text{o}(\bar{a})=4=\text{o}(\bar{b}).$ It is easy to see neither $\bar{a}=\bar{0}$ nor $\bar{b}=\bar{0}.$ If $\bar{a}=\bar{2},$ then $(\bar{2}, \bar{2})\sim (\bar{2}, \bar{b})$ implies there exists $k\in \mathbb{N}$ such that $k\bar{2}=\bar{2}\Rightarrow k\bar{2}=\bar{2}\Rightarrow k=4\ell+1\Rightarrow (4\ell+1)\bar{b}\neq \bar{2}.$ Similar result holds if $\bar{b}=\bar{2}.$ As a result, $(\bar{2}, \bar{2})\in \langle (\bar{1}, \bar{3})\rangle=\langle (\bar{3}, \bar{1})\rangle$ and $\langle (\bar{1}, \bar{1})\rangle=\langle (\bar{3}, \bar{3})\rangle.$

Let $G\cong\underbrace{\mathbb{Z}_p\times \cdots\times \mathbb{Z}_p}_{k \text{ times  }}.$ Here the graph  $\mathcal{P}^{**}(G)$ is isomorphic to the graph  $\underbrace{K_{\phi(p)}\bigoplus\cdots\bigoplus K_{\phi(p)}}_{k \text{ times }}.$ This completes the proof.
\end{proof}
\begin{proof} [Proof of Theorem \ref{thm: non cyclic abeln grp line graph of {P}^{**}(G)}]
Let $\mathcal{P}^{**}(G)$ be a line graph of some graph $\Gamma.$ Then by Proposition \ref{Prop: noncyclic, abln, {P}^{**}(G) does not contain Gamma_{1, 3}  }, we can say that either $G\cong \mathbb{Z}_2\times \mathbb{Z}_{2^2}$ or $G\cong \mathbb{Z}_{2^2}\times \mathbb{Z}_{2^2}$ or $G\cong \mathbb{Z}_p\times\cdots\times \mathbb{Z}_{p}.$ 
	
Conversely, we show that if $G$ is one of the above, then $\mathcal{P}^{**}(G)$ is line graph of some graph. If $G\cong \mathbb{Z}_2\times \mathbb{Z}_{2^2},$ then $\mathcal{P}^{**}(G)$ is the graph in Figure \ref{fig:pwr grapph for  group of order Z_2*Z_2^2} and it is the line graph of the graph $\Gamma,$ described as in Figure \ref{fig: G=Z_2*Z_2^2 P^**(G) is line grapph of graph }
	\begin{figure}[H]
		\tiny
		\centering
		\begin{tikzpicture}[scale=1]
		\tikzstyle{edge_style} = [draw=black, line width=2mm, ]
		\draw (0,0)--(0,2);
		\draw (0,2)--(-1,3.5);
		\draw (0,2)--(1,3.5);
		\draw (0,0)--(-1,-1.5);
		\draw (0,0)--(1, -1.5);
		\draw (-3,1)--(-1.5,1);
		\draw (1, 1)--(2.5,1);
		\node (e) at (-.3,0){$\bf{v_1}$};
		\node (e) at (-1.3,-1.5){$\bf{v_2}$};
		\node (e) at (1.3,-1.5){$\bf{v_3}$};
		\node (e) at (-.3, 2){$\bf{v_4}$};
		\node (e) at (1.3,3.5){$\bf{v_5}$};
		\node (e) at (-1.3,3.5){$\bf{v_6}$};
		\node (e) at (-3.3,1){$\bf{v_7}$};
		\node (e) at (-1.2,1){$\bf{v_8}$};
		\node (e) at (.7,1){$\bf{v_9}$};
		\node (e) at (2.9,1){$\bf{v_{10}}$};
		\node (e) at (-.5,1){$\bf{e(\bar{0}, \bar{2})}$};
		\node (e) at (-1.3,-.8){$\bf{e(\bar{0}, \bar{1})}$};
		\node (e) at (1.3,-.8){$\bf{e(\bar{0}, \bar{3})}$};
		\node (e) at (-1.3, 2.8){$\bf{e(\bar{1}, \bar{3})}$};
		\node (e) at (1.3, 2.8){$\bf{e(\bar{1}, \bar{1})}$};
		\node (e) at (1.8,1.2){$\bf{e(\bar{1}, \bar{2})}$};
		\node (e) at (-2.3,1.2){$\bf{e(\bar{1}, \bar{0})}$};
		\fill[black!100!] (0,0) circle (.05);
		\fill[black!100!] (0,2) circle (.05);
		\fill[black!100!] (-1, 3.5) circle (.05);
		\fill[black!100!] (-1,-1.5) circle (.05);
		\fill[black!100!] (-3,1) circle (.05);
		\fill[black!100!] (1,1) circle (.05);
		\filldraw[black!100] (2.5,1) circle (.05);
		\filldraw[black!100] (-1.5,1) circle (.05);
		\filldraw[black!100] (1,-1.5) circle (.05);
		\fill[black!100!] (1, 3.5) circle (.05);
		\end{tikzpicture}
		\caption{The graph $\Gamma$ such that $\mathcal{P}^{**}(\mathbb{Z}_2\times\mathbb{Z}_{2^2})=L(\Gamma).$ }
		\label{fig: G=Z_2*Z_2^2 P^**(G) is line grapph of graph }	
	\end{figure}
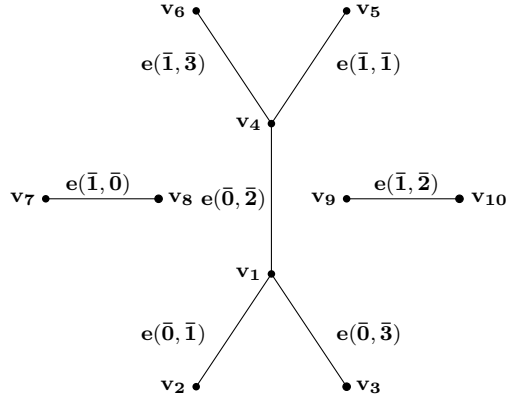
Let $G\cong\mathbb{Z}_{2^2}\times\mathbb{Z}_{2^2}.$ The group $G$ has $6$ distinct cyclic subgroups of order $4,$ namely 
$H_1=\langle (\bar{1}, \bar{0})\rangle, H_2=\langle (\bar{0}, \bar{1})\rangle, H_3=\langle (\bar{1}, \bar{1})\rangle,
H_4=\langle (\bar{1}, \bar{3})\rangle, H_5=\langle (\bar{1}, \bar{2})\rangle, H_6=\langle (\bar{2}, \bar{1})\rangle.$
Again $(\bar{2},\bar{0})\in H_1\cap H_5, (\bar{2},\bar{2})\in H_3\cap H_4, (\bar{0},\bar{2})\in H_2\cap H_6.$ Therefore, $\mathcal{P}^{**}(G)$ is the graph in Figure \ref{fig:pwr grapph for  group of  Z_2^2*Z_2^2}
	\begin{figure}[H]
		\tiny
		\centering
		\begin{tikzpicture}[scale=1]
		\tikzstyle{edge_style} = [draw=black, line width=2mm, ]
		\draw (0,0)--(2,0);
		\draw (0, 0)--(0,2);
		\draw (0,2)--(2,0);
		\draw (2,0)--(4,0);
		\draw (4,0)--(4,2);
		\draw (4,2)--(2,0);
		\node (e) at (0,-.3){$\bf{(\bar{1}, \bar{0})}$};
		\node (e) at (2,-.3){$\bf{(\bar{2}, \bar{0})}$};
		\node (e) at (0,2.2){$\bf{(\bar{3}, \bar{0})}$};
		\node (e) at (4,-.3){$\bf{(\bar{1}, \bar{2})}$};
		\node (e) at (4,2.2){$\bf{(\bar{3}, \bar{2})}$};
		\fill[black!100!] (0,0) circle (.05);
		\fill[black!100!] (2,0) circle (.05);
		\fill[black!100!] (0,2) circle (.05);
		\fill[black!100!] (4,0) circle (.05);
		\fill[black!100!] (4,2) circle (.05);
		\node (e) at (4.5,1){$\bf{\bigoplus}$};
		\draw (5,0)--(7,0);
		\draw (5, 0)--(5,2);
		\draw (5,2)--(7,0);
		\draw (7,0)--(9,0);
		\draw (9,0)--(9,2);
		\draw (9,2)--(7,0);
		%
		\node (e) at (7,-.3){$\bf{(\bar{0}, \bar{2})}$};
		\node (e) at (5,2.2){$\bf{(\bar{0}, \bar{1})}$};
		\node (e) at (5,-.3){$\bf{(\bar{0}, \bar{3})}$};
		\node (e) at (9,-.3){$\bf{(\bar{2}, \bar{1})}$};
		\node (e) at (9,2.2){$\bf{(\bar{2}, \bar{3})}$};
		\fill[black!100!] (5,0) circle (.05);
		\fill[black!100!] (7,0) circle (.05);
		\fill[black!100!] (5,2) circle (.05);
		\fill[black!100!] (9,0) circle (.05);
		\fill[black!100!] (9,2) circle (.05);
		\node (e) at (9.5,1){$\bf{\bigoplus}$};
		\draw (10,0)--(12,0);
		\draw (10, 0)--(10,2);
		\draw (10,2)--(12,0);
		\draw (12,0)--(14,0);
		\draw (14,0)--(14,2);
		\draw (14,2)--(12,0);
		
		\node (e) at (12,-.3){$\bf{(\bar{2}, \bar{2})}$};
		\node (e) at (10,2.2){$\bf{(\bar{1}, \bar{1})}$};
		\node (e) at (10,-.3){$\bf{(\bar{3}, \bar{3})}$};
		\node (e) at (14,-.3){$\bf{(\bar{3}, \bar{1})}$};
		\node (e) at (14,2.2){$\bf{(\bar{1}, \bar{3})}$};
		\fill[black!100!] (12,0) circle (.05);
		\fill[black!100!] (10,0) circle (.05);
		\fill[black!100!] (10,2) circle (.05);
		\fill[black!100!] (14,0) circle (.05);
		\fill[black!100!] (14,2) circle (.05);
		\end{tikzpicture}
		\caption{The graph $\mathcal{P}^{**}(\mathbb{Z}_{2^2}\times\mathbb{Z}_{2^2})$}
		\label{fig:pwr grapph for  group of  Z_2^2*Z_2^2}	
	\end{figure}
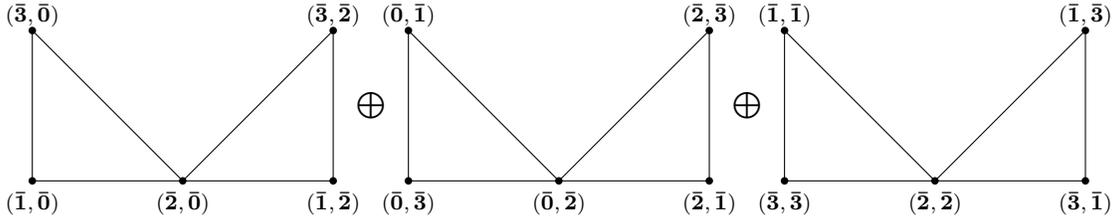
Clearly, $\mathcal{P}^{**}(G)$ is the line graph of the graph described as in Figure \ref{fig: G=Z_2^2*Z_2^2  line grapph of graph of P^**(G) } 
	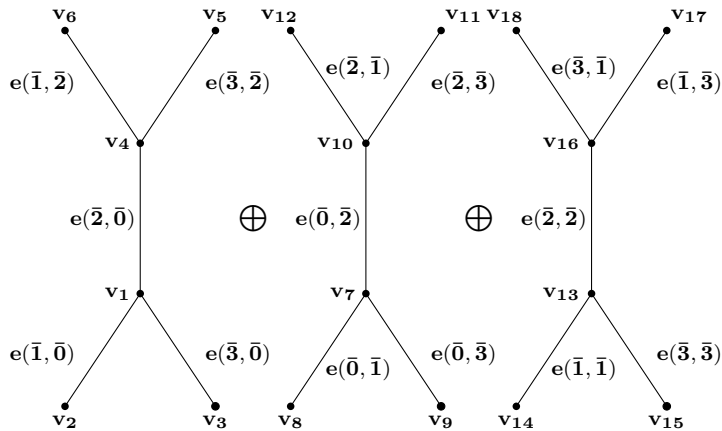
\begin{figure}[H]
		\tiny
		\centering
		\begin{tikzpicture}[scale=1]
		\tikzstyle{edge_style} = [draw=black, line width=2mm, ]
		\draw (0,0)--(0,2);
		\draw (0,2)--(-1,3.5);
		\draw (0,2)--(1,3.5);
		\draw (0,0)--(-1,-1.5);
		\draw (0,0)--(1, -1.5);
		\node (e) at (-.3,0){$\bf{v_1}$};
		\node (e) at (-1,-1.7){$\bf{v_2}$};
		\node (e) at (1,-1.7){$\bf{v_3}$};
		\node (e) at (-.3, 2){$\bf{v_4}$};
		\node (e) at (1,3.7){$\bf{v_5}$};
		\node (e) at (-1,3.7){$\bf{v_6}$};
		%
		\node (e) at (-.5,1){$\bf{e(\bar{2}, \bar{0})}$};
		\node (e) at (-1.3,-.8){$\bf{e(\bar{1}, \bar{0})}$};
		\node (e) at (1.3,-.8){$\bf{e(\bar{3}, \bar{0})}$};
		\node (e) at (-1.3, 2.8){$\bf{e(\bar{1}, \bar{2})}$};
		\node (e) at (1.3, 2.8){$\bf{e(\bar{3}, \bar{2})}$};
		%
		\fill[black!100!] (0,0) circle (.05);
		\fill[black!100!] (0,2) circle (.05);
		\fill[black!100!] (-1, 3.5) circle (.05);
		\fill[black!100!] (-1,-1.5) circle (.05);
		\filldraw[black!100] (1,-1.5) circle (.05);
		\fill[black!100!] (1, 3.5) circle (.05);
		\node (e) at (1.5,1){$\bf{\bigoplus}$};
		\draw (3,0)--(3,2);
		\draw (3,2)--(2,3.5);
		\draw (3,2)--(4,3.5);
		\draw (3,0)--(2,-1.5);
		\draw (3,0)--(4, -1.5);
		\node (e) at (2.7,0){$\bf{v_7}$};
		\node (e) at (2,-1.7){$\bf{v_8}$};
		\node (e) at (4,-1.7){$\bf{v_{9}}$};
		\node (e) at (2.6, 2){$\bf{v_{10}}$};
		\node (e) at (4.3,3.7){$\bf{v_{11}}$};
		\node (e) at (1.85,3.7){$\bf{v_{12}}$};
		%
		\node (e) at (2.5,1){$\bf{e(\bar{0}, \bar{2})}$};
		\node (e) at (2.9,-1){$\bf{e(\bar{0}, \bar{1})}$};
		\node (e) at (4.3,-.8){$\bf{e(\bar{0}, \bar{3})}$};
		\node (e) at (2.9, 3){$\bf{e(\bar{2}, \bar{1})}$};
		\node (e) at (4.3, 2.8){$\bf{e(\bar{2}, \bar{3})}$};
		\fill[black!100!] (3,0) circle (.05);
		\fill[black!100!] (3,2) circle (.05);
		\fill[black!100!] (2, 3.5) circle (.05);
		\fill[black!100!] (2,-1.5) circle (.05);
		\filldraw[black!100] (4,-1.5) circle (.05);
		\fill[black!100!] (4, 3.5) circle (.05);
		\node (e) at (4.5,1){$\bf{\bigoplus}$};;
		\draw (6,0)--(6,2);
		\draw (6,2)--(5,3.5);
		\draw (6,2)--(7,3.5);
		\draw (6,0)--(5,-1.5);
		\draw (6,0)--(7, -1.5);
		\node (e) at (5.6,0){$\bf{v_{13}}$};
		\node (e) at (5,-1.7){$\bf{v_{14}}$};
		\node (e) at (7,-1.7){$\bf{v_{15}}$};
		\node (e) at (5.6, 2){$\bf{v_{16}}$};
		\node (e) at (7.3,3.7){$\bf{v_{17}}$};
		\node (e) at (4.85,3.7){$\bf{v_{18}}$};
		%
		\node (e) at (5.5,1){$\bf{e(\bar{2}, \bar{2})}$};
		\node (e) at (5.9,-1){$\bf{e(\bar{1}, \bar{1})}$};
		\node (e) at (7.3,-.8){$\bf{e(\bar{3}, \bar{3})}$};
		\node (e) at (5.9, 3){$\bf{e(\bar{3}, \bar{1})}$};
		\node (e) at (7.3, 2.8){$\bf{e(\bar{1}, \bar{3})}$};
		%
		\fill[black!100!] (6,0) circle (.05);
		\fill[black!100!] (6,2) circle (.05);
		\fill[black!100!] (5, 3.5) circle (.05);
		\fill[black!100!] (5,-1.5) circle (.05);
		\filldraw[black!100] (7,-1.5) circle (.05);
		\fill[black!100!] (7, 3.5) circle (.05);
		\end{tikzpicture}
		\caption{The graph $\Gamma$ such that $\mathcal{P}^{**}(\mathbb{Z}_{2^2}\times\mathbb{Z}_{2^2})=L(\Gamma).$}
		\label{fig: G=Z_2^2*Z_2^2  line grapph of graph of P^**(G) }	
	\end{figure}
Let $G\cong\underbrace{\mathbb{Z}_p\times \cdots\times \mathbb{Z}_p}_{k \text{ times}}.$ In this case $\mathcal{P}^{**}(G)\cong \underbrace{K_{\phi(p)}\bigoplus\cdots\bigoplus K_{\phi(p)}}_{k \text{ times }}.$ Clearly, it is the line graph of the graph $\underbrace{\Gamma_{1, \phi(p)}\bigoplus\cdots\bigoplus\Gamma_{1, \phi(p)}}_{k \text{ times }},$ where $k=p^{k-1}+p^{k-2}+\cdots+p+1.$ This completes the proof.
\end{proof}
In this portion we study the non ablien nilpotent groups $G$ for which $\mathcal{P}^{**}(G)$ is line graph. 
\begin{theorem}\label{thm:  non abelian G has geq 3 distinct prime divisor, not line graph }
Let $G$ be a non abelian nilpotent group such that $|G|$ has at least three distinct prime divisors. Then there is no graph $\Gamma$ such that $\mathcal{P}^{**}(G)=L(\Gamma).$	
\end{theorem}
\begin{proof}
Let $p_1, p_2, p_3$ be three distinct prime divisors of $|G|.$ Now $G$ nilpotent implies that $G$ has an element $v^{(p_1p_2p_3)}$ such that $\text{o}(v)=p_1p_2p_3.$ Now $\langle v^{(p_1p_2p_3)}\rangle$ has elements $v^{(p_1)}, v^{(p_2)}, v^{(p_3)}$ such that $\text{o}(v^{(p_1)})=p_1, \text{o}(v^{(p_2)})=p_2, \text{o}(v^{(p_3)})=p_3.$ Clearly, $\mathcal{P}^{**}(G)$ has an induced subgraph $\Gamma_{1, 3}(v^{(p_1p_2p_3)}, v^{(p_1)}, v^{(p_2)}, v^{(p_3)}).$ Hence the theorem.
\end{proof}
\begin{theorem}\label{thm: G non ab nilpotent 2 prime divisors, P**(G) not a line graph}
Let $G$ be non abelian nilpotent group such that $|G|$ has exactly two distinct prime divisors. Then there is no graph $\Gamma$ such that $\mathcal{P}^{**}(G)=L(\Gamma).$ 	
\end{theorem}
\begin{proof}
It is given that $G$ is non abelian nilpotent and $|G|$ has two distinct prime divisors. Let the prime divisors are $p_1$ and $p_2.$ Note that $G\cong P_1\times P_2,$ where $P_1, P_2$ are Sylow subgroups of $G$ with $|P_1|=p_1^{\alpha_1}$ and $|P_2|=p_2^{\alpha_2}.$ So $|G|=p_1^{\alpha_1}p_2^{\alpha_2},$ where at least one $\alpha_i \geq 2(i=1, 2).$ Otherwise, $G$ would be cyclic. First suppose that $p_1$ and $p_2$ are odd primes. Now $G$ is nilpotent, therefore $G$ has an element $v^{(p_1p_2)}$ of order $p_1p_2.$ Consider the cyclic subgroup $H=\langle v^{(p_1p_2)}\rangle.$ (Note that all the elements of the cyclic group $H$ belong to $V(\mathcal{P}^{**}(G)).)$ Again $H$ has elements $v^{(p_1)}$ and $v^{(p_2)}$ of order $p_1$ and $p_2$ respectively. Also no elements in $\langle v^{(p_1)}\rangle $ is edge connected with the element in $\langle v^{(p_2)}\rangle.$ Now $p_1$ and $p_2$ are odd primes, which imply that $\phi(p_1p_2), \phi(p_1), \phi(p_2)\geq2.$ Let $v_1^{(p_1)}, v_2^{(p_1)}, v_1^{(p_2)} v_2^{(p_2)} \text{ and } v_1^{(p_1p_2)}, v_1^{(p_1p_2)}$ are elements of $H$ such that
$\text{o}(v_1^{(p_1)})=\text{o}(v_2^{(p_1)})=p_1, \text{o}(v_1^{(p_2)})=\text{o}(v_2^{(p_2)})=p_2, \text{ and } \text{o}(v_1^{(p_1p_2)})=\text{o}(v_2^{(p_1p_2)})=p_1p_2.$
Now we replace the vertices of the graph $\Gamma_3$ in Figure \ref{fig:line grapph theory} in the following way:
\[11 \text{ by } v_1^{(p_1p_2)}, 13 \text{ by } v_2^{(p_1p_2)}, 12 \text{ by }v_1^{(p_1)}, 15 \text{ by } v_2^{(p_1)}, 14 \text{ by } v_1^{(p_2)} \text{ and } 16 \text{ by } v_2^{(p_2)}.\]
It is easy to see that the resulting graph is isomorphic to the graph $\Gamma_3$ in Figure \ref{fig:line grapph theory}. 

Now let at least one of $p_1$ and $p_2$ be $2.$ Let $p_1=2.$ Then $|G|=2^kp_2^r,$ for some $r, k\in \mathbb{N}$ with at least one of $r, k\geq 2.$ First suppose that $G$ has an element of order $2^k, k\geq2.$ Now $G$ has an element $v^{(2^kp_2)}$ such that $\text{o}(v^{(2^kp_2)})=2^kp_2.$ As $\phi(2^kp_2)\geq 2,$ we can choose two elements $v_1^{(2^kp_2)}, v_2^{(2^kp_2)}$ form $\langle v^{(2^kp_2)}\rangle$ such that $\text{o}(v_1^{(2^kp_2)})=2^kp_2$ and $\text{o}(v_1^{(2^kp_2)})=2^kp_2.$ Also $\langle v^{(2^kp_2)}\rangle$ has elements namely, $v_1^{(2^k)}, v_2^{(2^k)} (\text{ as }k\geq 2, \phi(2^k)\geq2), v_1^{(p_2)}, v_2^{(p_2)}$ such that $\text{o}(v_1^{(2^k)})=2^k, \text{o}(v_2^{(2^k)})=2^k, \text{o}(v_1^{(p_2)})=p_2, \text{o}(v_2^{(p_2)})=p_2$ ( $p_2$ is an odd prime, so it is possible to find at least two elements of order $p_2$ in $\langle v^{(2p_2)}\rangle).$ 
Now we replace the vertices of the graph $\Gamma_3$ in Figure \ref{fig:line grapph theory} in the following way:
\[11 \text{ by } v_1^{(2^kp_2)}, 13 \text{ by } v_2^{(2^kp_2)}, 12 \text{ by }v_1^{(2^k)}, 15 \text{ by } v_2^{(2^k)}, 14 \text{ by } v_1^{(p_2)} \text{ and } 16 \text{ by } v_2^{(p_2)}.\] As a result $\mathcal{P}^{**}(G)$ has an induced subgraph isomorphic to the graph $\Gamma_3.$ Therefore, in this case $\mathcal{P}^{**}(G)$ is not a line graph.
 
Now suppose that order of each element of the Sylow subgroup $P_1$ is $2.$ Then either $P_1\cong \underbrace{\mathbb{Z}_2\times \cdots \times \mathbb{Z}_2}_{k(\geq2)\text{ times }}$ or $P_1\cong \mathbb{Z}_2.$ Therefore, either $G\cong\underbrace{\mathbb{Z}_2\times \cdots \times \mathbb{Z}_2}_{k(\geq2)\text{ times }}\times P_2$ or $G\cong \mathbb{Z}_2\times P_2.$ In both of the cases, $G$ has at least one element $v^{(2)}$ (say) of order $2.$ Let $H_1=\langle v_1^{(p_2)}\rangle, \cdots, H_{\ell}=\langle v_{\ell}^{(p_2)}\rangle$ be the complete list of distinct cyclic subgroups of order $p_2$ of the Sylow subgroup $P_2.$  Now by the claim in Case 3 of the proof of Theorem \ref{thm: P(G) line graph, G non abelian nilpotent} $\ell\neq 2.$ Again $\ell=1$ implies that $G$ is abelian by Lemma \ref{p group unique subgrp of order p, g is cyclic}. Therefore, $\ell\geq 3.$ Consider $K_1=\langle v^{(2)}v_1^{(p_2)} \rangle, \cdots, K_{\ell}=\langle v^{(2)}v_{\ell}^{(p_2)} \rangle.$ Clearly, each $K_i$ is a cyclic group of order $2p_2.$  Moreover, for all $i\neq j, K_i\neq K_j$ and $K_i\cap K_j=\langle v^{(2)}\rangle,$ where $i, j\in \{1, \cdots, \ell\}, (\ell\geq 3).$ Now we choose three generators $ v^{(2)}v_{i_1}^{(p_2)}, v^{(2)}v_{i_2}^{(p_2)} \text{ and } v^{(2)}v_{i_3}^{(p_2)}$ from three distinct cyclic subgroups $K_{i_1}, K_{i_2}$ and $K_{i_3}$ respectively, where $i_1, i_2, i_3\in \{1, \cdots, \ell\}.$ Since $K_{i_1}, K_{i_2}$ and $K_{i_3}$ are distinct cyclic subgroups, then $ v^{(2)}v_{i_1}^{(p_2)}, v^{(2)}v_{i_2}^{(p_2)} \text{ and } v^{(2)}v_{i_3}^{(p_2)}$ are not adjacent to reach other in $\mathcal{P}^{**}(G).$ Also $\langle v^{(2)}\rangle$ is contained in each of the cyclic subgroups $K_{i_1}, K_{i_2}$ and $K_{i_3}.$ As a result, $\mathcal{P}^{**}(G)$ contains an induced subgraph $\Gamma_{1, 3}(v^{(2)}, v^{(2)}v_{i_1}^{(p_2)}, v^{(2)}v_{i_2}^{(p_2)}, v^{(2)}v_{i_3}^{(p_2)}).$ This completes the proof. 
\end{proof}
\begin{theorem}\label{Thm: non abelian p group and line graph possibility of P**(G)}
Let $G$ be a non abelian $p$-group, where $p$ is an odd prime. Then there is a graph $\Gamma$ such that $\mathcal{P}^{**}(G)=L(\Gamma)$ if and only if $G=\mathbb{Z}_{p^{t_1}}\cup \cdots\cup \mathbb{Z}_{p^{t_{\ell}}},$ where $\ell$ is the number of distinct subgroups of order $p.$ 	
\end{theorem}
\begin{proof}
First suppose that $G=\mathbb{Z}_{p^{t_1}}\cup \cdots\cup \mathbb{Z}_{p^{t_{\ell}}},$ where $\ell$ is the number of distinct subgroups of order $p.$ Now for any $i\neq j,$ $\mathbb{Z}_{p^{t_i}}\cap \mathbb{Z}_{p^{t_j}}=\{e\},$ the identity of the group $G.$ In fact, for a non identity element $a\in \mathbb{Z}_{p^{t_i}}\cap \mathbb{Z}_{p^{t_j}},$ the $p$-order cyclic subgroup of $\langle a\rangle$ is contained in $ \mathbb{Z}_{p^{t_i}}\cap \mathbb{Z}_{p^{t_j}}.$ And this contradicts that $G=\mathbb{Z}_{p^{t_1}}\cup \cdots\cup \mathbb{Z}_{p^{t_{\ell}}},$ where $\ell$ is the number of distinct subgroups of order $p.$ As a result, for any $i\neq j,$ $\mathbb{Z}_{p^{t_i}}\cap \mathbb{Z}_{p^{t_j}}=\{e\}.$ So, $\mathcal{P}^{**}(G)$ is a line graph of the graph $\Gamma_{1, p^{t_1}}\bigoplus\cdots \bigoplus\Gamma_{1, p^{t_{\ell}}}.$

Conversely, suppose that there is a graph $\Gamma$ such that $\mathcal{P}^{**}(G)=L(\Gamma).$ We show that $G=\mathbb{Z}_{p^{t_1}}\cup \cdots\cup \mathbb{Z}_{p^{t_{\ell}}},$ where $\ell$ is the number of distinct subgroups of order $p.$ Now $G$ is a $p$-group. So we can write $G=K_1\cup\cdots \cup K_r,$ where 
\begin{align*}
K_1&=\{x_{(1)}\in G: \langle a_{(1)}\rangle \subset \langle x_{(1)}\rangle \text{ and } \text{o}(a_{(1)})=p\}\cup \{e\}\\
K_2&=\{x_{(2)}\in G: \langle a_{(2)}\rangle \subset \langle x_{(2)}\rangle, a_{(2)}\in G\setminus \langle a_{(1)}\rangle \text{ and }\text{o}(a_{(2)})=p\}\cup \{e\}\\
\vdots& \hspace{40mm} \vdots\\
K_r&=\{x_{(r)}\in G: \langle a_{(r)}\rangle \subset \langle x_{(r)}\rangle, a_{(r)}\in G\setminus \langle a_{(1)}\rangle\cup\cdots\cup\langle a_{(r-1)}\rangle, \text{ and } \text{o}(a_{(r)})=p\}\cup \{e\}
\end{align*}
(Note that $r\neq1.$ If $r=1,$ then $G$ is a group with unique minimal subgroup. Again $p$ is an odd prime, therefore, $G$ is a cyclic $p$-group, it contradicts that $G$ is non abelian). Clearly, $K_i\cap K_j=\{e\}$ for any $i\neq j.$
Now it is enough to show that each $K_i=\mathbb{Z}_{p^{t_i}},$ for some $t_i\geq 1.$ Let there is at least one $j$ such that $K_j$ is not a cyclic subgroup of $G.$ Then we can write $K_j$ as a union cyclic subgroups of $G.$ In fact, since $K_j$ is not a cyclic group there exists at least one element say $x_{(j)}^{(p^r)}\in K_j$ such that $\text{o}(x_{(j)}^{(p^r)})=p^{r}, r\geq 2,$ (if order of each element of $K_j$ is $p$ then $K_j$ would be cyclic by the construction of $K_j).$ Consider $K_{(j)}^{r}=\langle x_{(j)}^{(p^r)}\rangle.$ Clearly, it is a cyclic subgroup of $K_j.$ Since $K_j$ is not cyclic group then there is an element say $\tilde{x_{(j)}^{(p^r)}}\in K_j\setminus K_{(j)}^{r}$ such that $\text{o}(\tilde{x_{(j)}^{(p^r)}})=p^s, s\geq 2 (s=1 \text{ implies } K_j=\mathbb{Z}_{p^r}).$ Let $K_{(j)}^s=\langle \tilde{x_{(j)}^{(p^s)}}\rangle.$ Note that neither $K_{(j)}^r$ is a subgroup of $ K_{(j)}^s$ nor $K_{(j)}^s$ is a subgroup of $K_{(j)}^r$ and $K_{(j)}^r\cap K_{(j)}^s$ contains the cyclic subgroup $\langle a_{(j)}\rangle.$ Now $p\geq 3$ implies that $\phi(p),\phi(p^r)\text{ and }\phi(p^s)\geq 2.$ Let $a_{(j)1}, a_{(j)2}$ be two $p$-ordered elements of $K_j, x^{(p^r)}_{(j)1}, x^{(p^r)}_{(j)2}$ be two $p^r$-ordered elements of $K_{(j)}^r$ and $x^{(p^s)}_{(j)1}, x^{(p^s)}_{(j)2}$ be two $p^s$-ordered elements of $K_{(j)}^s.$  Now we replace the vertices of the graph $\Gamma_3$ in Figure \ref{fig:line grapph theory} in the following way:
\[11 \text{ by }a_{(j)1},  13 \text{ by }a_{(j)2}, 12 \text{ by }x^{(p^r)}_{(j)1}, 15 \text{ by } x^{(p^r)}_{(j)2}, 14 \text{ by } x^{(p^s)}_{(j)1} \text{ and } 16 \text{ by }x^{(p^s)}_{(j)2}.\] Clearly the resulting graph is isomorphic to the graph $\Gamma_3.$ Therefore, $\mathcal{P}^{**}(G)$ has an induced subgraph isomorphic to $\Gamma_3.$ This contradicts the fact that $\mathcal{P}^{**}(G)$ is a line graph. Therefore, each $K_i=\mathbb{Z}_{p^{t_i}},$ for some $t_i\in \mathbb{N}.$ Clearly $r=\ell$ and hence $G=\mathbb{Z}_{p^{t_1}}\cup \cdots\cup\mathbb{Z}_{p^{t_{\ell}}},$ where $\ell$ is the number of distinct cyclic subgroup of order $p.$   
\end{proof}
\begin{proof}[Proof of Theorem \ref{Thm:P**(G) is line graph for nilpotent group}]
Clearly, Theorems \ref{classify: G cyclic line graph, P^{**}(G)} \ref{thm: non cyclic abeln grp line graph of {P}^{**}(G)}, \ref{thm:  non abelian G has geq 3 distinct prime divisor, not line graph }, \ref{thm: G non ab nilpotent 2 prime divisors, P**(G) not a line graph}, \ref{Thm: non abelian p group and line graph possibility of P**(G)} complete the proof.	
\end{proof}
Now we concentrate on non abelian $2$-group. For that we need the structures of dihedral groups and generalized quarternion groups.
For $n \geq 2$, the \emph{dihedral group} of order $2n$ is defined by the following presentation:
\[ D_{2n}= \langle r, s : r^n=s^2=e, rs=sr^{-1} \rangle.\]
We also consider the generalized \emph{quarternion groups} $Q_{2^n}.$ Let $x = \overline{(1, 0)}$ and $y = \overline{(0, 1)}.$ Then $Q_{2^n} = \langle x, y\rangle,$ where
\begin{enumerate}
	\item[(a)]
	$x$ has order $2^{n-1}$ and $y$ has order $4,$
	\item[(b)]
	every element of $Q_{2^n}$ can be written in the form $x^a$ or $x^ay$ for some $a\in \mathbb{Z},$
	\item[(c)]
	$x^{2^{n-2}}=y^2,$
	\item[(d)]
	for each $g\in Q^{2^n}$ such that $g\in \langle x \rangle,$ such that $gxg^{-1}=x^{-1}.$
\end{enumerate}
For more information about $D_{2n},$ and $Q_{2^n}$ see \cite{generalized-quaternion, algebradummitfoote, scott-group}.
\begin{proof} [Proof of Theorem \ref{THm: Line graph of generalized quaternion group}]
We know that $Q_{2^n}$ has exactly one cyclic subgroup $H$ of order $2^{n-1}.$ Also each element in $Q_{2^n}\setminus H$ is of order $4.$ So, there are $2^{n-2}$ distinct $4$-ordered cyclic subgroups in $Q_{2^n}.$ Clearly, $\mathcal{P}^{**}(Q_{2^n})$ is the graph $K_{2^{n-1}}\bigoplus \underbrace{K_2\bigoplus\cdots\bigoplus K_2}_{ 2^{n-2}\text{ times}}.$ Therefor, $\mathcal{P}^{**}(Q_{2^n})$ is the line graph of the graph $\Gamma_{1, 2^{n-1}}\bigoplus\underbrace{\Gamma_{1, 2}\bigoplus\cdots\bigoplus\Gamma_{1, 2}}_{ 2^{n-2}\text{ times}}.$ 
\end{proof}
\begin{proof}[Proof of Theorem \ref{thm: P**(D_n) is a line graph if and only}]
Suppose $n$ is not a power of $2.$  Then clearly $|D_n|$ has a prime divisor $p\neq 2.$ Therefore by Theorem \ref{thm: G non ab nilpotent 2 prime divisors, P**(G) not a line graph} $\mathcal{P}^{**}(D_n)$ is not a line graph.
	
Conversely, let $n=2^k,$ then $D_n$ is $2$-group. Also $D_n=\mathbb{Z}_{2^k}\cup \underbrace{\mathbb{Z}_2\cup\cdots\cup\mathbb{Z}_2}_{2^k \text{ times }}$ and the number of distinct $2$-ordered cyclic subgroup is $2^{k+1}.$ Then by Theorem \ref{Thm:P**(G) is line graph for nilpotent group}, $\mathcal{P}^{**}(D_n)$ is line graph.  	
\end{proof}

\subsection*{Acknowledgment} The author profusely thank the anonymous referee for meticulous reading of the manuscript and valuable suggestions that significantly improved the exposition of this paper. The author  would like to thank Prof. Arvind Ayyer for his constant support and encouragement. Also The author  would like to thank Dr. Sumana Hatui for the helpful discussions on $p$-groups. The author was supported by NBHM Post Doctoral Fellowship grant 0204/52/2019/RD-II/339. 
\bibliographystyle{amsplain}

\end{document}